\documentclass{amsart}

 \usepackage{amscd,amsfonts,amssymb,amsmath}
 \usepackage{graphicx,pst-all}
 \usepackage{caption}
 \usepackage{float}
 \usepackage{color}
 \usepackage{enumerate}
 \usepackage{mathrsfs}
 \usepackage[active]{srcltx}
 \usepackage{verbatim}
 \usepackage[colorlinks,linkcolor={black},citecolor={blue},urlcolor={black}]{hyperref}
 
\theoremstyle{plain}
\newtheorem{theorem}{Theorem}[section]
\newtheorem{corollary}[theorem]{Corollary}
\newtheorem{lemma}[theorem]{Lemma}
\newtheorem{proposition}[theorem]{Proposition}

\newtheorem{definition}[theorem]{Definition}

\newtheorem*{definition*}{Definition}

\theoremstyle{remark}
\newtheorem{remark}[theorem]{Remark}

\newtheorem{claim}{Claim}
\newtheorem*{claim*}{Claim}
\newtheorem*{remark*}{Remark}
\newtheorem*{example*}{Example}
\newtheorem*{notation*}{Notation}

\numberwithin{equation}{section}

\def\N{{\mathbb N}}

\def\R{{\mathbb R}}

\renewcommand{\P}{{\mathbb P}}

\newcommand{\1}{{{\bf 1}}}

\renewcommand{\a}{\alpha}

\renewcommand{\d}{\delta}
\newcommand{\eps}{\varepsilon}
\renewcommand{\phi}{\varphi}

\newcommand{\dd}{\; \mathrm{d}}
\newcommand{\xx}{\mathbf{x}}

\DeclareMathOperator{\supp}{supp}

\newcommand{\ip}[1]{\langle {#1}\rangle}

\newcommand{\norm}[1]{\| {#1}\|}
\newcommand{\abs}[1]{\vert {#1}\vert}

\DeclareMathOperator{\Ric}{Ric}
\DeclareMathOperator{\Lip}{\mathsf{Lip}}
\DeclareMathOperator{\Hess}{Hess}
\DeclareMathOperator{\Leb}{Leb}
\DeclareMathOperator{\ent}{Ent}

\DeclareMathOperator{\vol}{vol}

\DeclareMathOperator{\bB}{\boldsymbol{B}}

\newcommand{\ddt}{\frac{\mathrm{d}}{\mathrm{d}t}}
\newcommand{\ddtr}{\frac{\mathrm{d}^+}{\mathrm{d}t}}

\newcommand{\cC}{\mathcal{C}}

\newcommand{\cB}{\mathcal{B}}

\newcommand{\cE}{\mathcal{E}}
\newcommand{\cA}{\mathcal{A}}

\newcommand{\cF}{\mathcal{F}}

\newcommand{\cP}{\mathscr{P}}

\newcommand{\e}{\mathrm{e}}

\DeclareSymbolFont{AMSb}{U}{msb}{m}{n}

\newcommand{\cs}{\Upsilon}
\newcommand{\dcs}{d_\Upsilon}
\newcommand{\cyl}{\mathsf{Cyl}^\infty(\cs)}

\newcommand{\Ent}{\mbox{Ent}}

\newcommand{\Opt}{\mathsf{Opt}}
\newcommand{\pr}{\mbox{proj}}
\newcommand{\Ch}{\mathsf{Ch}}
\newcommand{\Cpl}{\mathsf{Cpl}}

\begin{document}

\title{Curvature bounds for configuration spaces}
\author{Matthias Erbar}
\author{Martin Huesmann}
\thanks{This material is based upon work supported by the National Science Foundation under Grant No.\ 0932078 000, while both authors were in residence at the Mathematical Sciences Research Institute in Berkeley, California, in the fall of 2013. M.E.\ gratefully acknowledges funding  from the European Research Council
under the European Community’s Seventh Framework Programme (FP7/2007-2013) / ERC Grant
agreement GeMeThnES No.\ 246923; M.H.\  from the CRC 1060. }
\address{
University of Bonn\\
Institute for Applied Mathematics\\
Endenicher Allee 60\\
53115 Bonn\\
Germany}
\email{erbar@iam.uni-bonn.de}
\email{huesmann@iam.uni-bonn.de}




\begin{abstract}
  We show that the configuration space $\cs$ over a manifold $M$
  inherits many curvature properties of the manifold. For instance, we
  show that a lower Ricci curvature bound on $M$ implies a lower Ricci
  curvature bound on $\cs$ in the sense of Lott--Sturm--Villani, the
  Bochner inequality, gradient estimates and Wasserstein
  contraction. Moreover, we show that the heat flow on $\cs$ can be
  identified as the gradient flow of the entropy.
\end{abstract}

\maketitle


\section{Introduction}

The configuration space $\cs$ over a manifold $M$ is the space of all
locally finite point measures, i.e.\
$$ \cs=\{\gamma\in\mathcal M(M): \gamma(K)\in \N_0 \mbox{ for all compact } K\subset M\}.$$
In the seminal paper \cite{AKR98} {\sc Albeverio--Kondratiev--R\"ockner}
identified a natural geometry on $\cs$ by ``lifting'' the geometry of
$M$ to $\cs$. In particular, there is a natural gradient $\nabla^\cs$,
divergence $\mbox{div}^\cs$ and Laplace operator $\Delta^\cs$ on the configuration
space. It turns out that the Poisson measure $\pi$ is the unique (up
to the intensity) measure on $\cs$ under which the gradient and
divergence become dual operators in $L^2(\cs,\pi).$ Hence, the Poisson
measure is the natural volume measure on $\cs$ and $\cs$ can be seen
as an infinite dimensional Riemannian manifold. The canonical Dirichlet form
$$\cE(F)=\int_\cs |\nabla^\cs F|^2_\gamma \ \pi(d\gamma)$$
induces the heat semigroup $T_t^\cs$ and a Brownian motion on $\cs$
which can be identified with the independent infinite particle
process.  The intrinsic metric $\dcs(\gamma,\omega)$ between two
configurations $\gamma$ and $\omega$ with respect to $\cE$ is the
non-normalized $L^2$ Wasserstein distance between the two measures
$\gamma$ and $\omega$. Typically, $\dcs$ will attain the value
$\infty$.
\medskip\\
In this article, we are interested in the curvature of $\cs$. We will
not try to define a curvature tensor. Instead, we will show that many
analytic and geometric estimates that characterize lower curvature
bounds on Riemannian manifolds \emph{lift} to natural analogues on the
configuration space.

We will first consider \emph{sectional curvature}. There are many
equivalent ways of characterizing a global lower bound $K\in\R$ on the
sectional curvature using only the Riemannian distance $d$,
e.g. Toponogov's Theorem on triangle comparison. This allows
to define a generalized sectional curvature bound also for metric
spaces leading to the notion of Alexandrov geometry, we point the
reader to \cite{BBI} for a detailed account. Our first result is that
sectional curvature bounds lift from $M$ to $\cs$.

\begin{theorem}\label{thm:main alex}
  If $M$ has sectional curvature bounded below by $K\in\R$
  then the configuration space $\cs$ has Alexandrov
  curvature bounded below by $\min\{K,0\}$.
\end{theorem}

From now on we will be concerned with lower bounds on the \emph{Ricci
  curvature}. They allow to control various analytic, stochastic and
geometric quantities, like the volume growth and the heat kernel. A
uniform lower bound $\Ric \geq K$ can be encoded in many different
ways. Let us recall some of them.

\begin{itemize}
\item[(BI)] \emph{Bochner's inequality:} for every smooth function
  $u:M\to\R$
  \begin{align*}
    \frac12\Delta\abs{\nabla u}^2 -\ip{\nabla u,\nabla\Delta u} \geq K
    |\nabla u|^2\;.
  \end{align*}

\item[(GE)] \emph{Gradient estimate:} for every smooth function $u$
  \begin{align*}
    |\nabla T^M_t u|^2\leq \e^{-2Kt}T^M_t|\nabla u|^2\;.
  \end{align*}
\end{itemize}

Here $T^M_t=\e^{t\Delta}$ denotes the heat semigroup on $M$. (BI) is
easily seen to be equivalent to $\Ric\geq K$ by noting that the left
hand side equals $\Ric[\nabla u] + \norm{\Hess u}^2_{HS}$. The
equivalence of (BI) and (GE) is due to a classic interpolation argument
of {\sc Bakry--\'Emery} \cite{BE85}.
 
Other ways of encoding a lower Ricci bound involve the action of the
(dual) heat semigroup on probability measures and the
$L^2$-transportation distance between probability measures. For
$\mu\in\cP(M)$ the probability measure $H^M_t\mu$ is defined via $\int
f \dd H^M_t\mu = \int T^M_tf \dd\mu$. Given $\mu_0,\mu_1\in\cP(M)$
their $L^2$-transportation distance associated to the Riemannian
distance $d$ is defined by
\begin{align*}
  W^2_{2,d}(\mu_0,\mu_1)~=~\inf\left\{\int d^2(x,y)\dd q(x,y)\right\}\;,
\end{align*}
where the infimum is taken over all couplings of $\mu_0,\mu_1$. Recall
also the relative entropy of a measure $\mu=\rho m$ w.r.t.\ the volume
measure $m$ given by $\ent(\mu|m)=\int\rho\log\rho\dd m$. Then, a lower
bound $\Ric\geq K$ is equivalent to

\begin{itemize}
\item[(WC)] \emph{$W_{2,d}$-contractivity:} for all $\mu_0,\mu_1\in\cP(M)$ and $t>0$:
  \begin{align*}
    W_{2,d}(H^M_t\mu,H^M_t\nu)~\leq~\e^{-Kt}W_{2,d}(\mu,\nu)\;.
  \end{align*}
\item[(GC)] \emph{Geodesic convexity of $\ent$:} for every
  (constant-speed) geodesic $(\mu_t)_{t\in[0,1]}$ in $(\cP(M),W_{2,d})$ and all $t\in[0,1]$:
  \begin{align*}
    \ent(\mu_s|m)\leq (1-s) \ent(\mu_0|m) + s\ent(\mu_1|m) - \frac{K}{2s(1-s)}W^2_{2,d}(\mu_0,\mu_1)\;.
  \end{align*}
\end{itemize}

These equivalences have been established in
\cite{vRenesseSturm:gradientestimates,CMS01}. Finally, (WC) and (GC)
can be captured in a single inequality

\begin{itemize}
\item[(EVI)] \emph{Evolution Variational Inequality:} for all $\mu,\sigma\in \cP(M)$ with finite
second moment and a.e. $t>0$:
\begin{align*}
  \ddt \frac12 W_{2,d}^2(H^M_t\mu,\sigma) + \frac{K}2
  W_{2,d}^2(H^M_t\mu,\sigma)~\leq~\ent(\sigma|m)-\ent(H^M_t\mu|m)\;.
\end{align*}
\end{itemize}

The last property (EVI) was first established in the Riemannian
setting in \cite{OW05,DS08}. It is also a way of stating that the heat
flow is the gradient flow of the entropy in the metric space
$(\cP(M),W_{2,d})$ and thus a reformulation of the celebrated result by
{\sc Jordan--Kinderlehrer--Otto} \cite{JKO}.

Notably, the property (GC) does not use the differential structure of
$M$ and can be formulated in the framework of metric measure
spaces. {\sc Sturm} \cite{sturm:gmms1} and {\sc Lott--Villani}
\cite{LV09} used this observation to define a notion of lower Ricci
curvature bound for metric measure spaces. The stronger property (EVI)
was studied on metric measure spaces in a series of papers by {\sc
  Ambrosio--Gigli--Savar\'e} \cite{AGS11a,AGS11b,AGS12}. There the
authors show the equivalence of (EVI) with suitable weak forms of (BI)
and (GE) for the canonical linear heat flow on such spaces.
\medskip\\
Unfortunately, most of this theory does not apply to the configuration
space since $(\cs,\dcs,\pi)$ is only an extended metric measure space,
the distance $\dcs$ can attain the value $\infty$. However, due to the
rich structure of $\cs$ we can establish suitable analogues of the
various manifestations of Ricci bounds.

Denote by $T^\cs_t=\e^{t\Delta^\cs}$ the heat semigroup on the
configuration space. For an absolutely continuous probability measure
$\mu\in\cP(\cs)$ with $\mu=f\pi$ define the dual semigroup
$H^\cs_t\mu=(T^\cs_t f)\pi$. Moreover, let now denote $W_{2,\dcs}$ the
$L^2$-transportation distance on $\cP(\cs)$ built from $d_\cs$. The
domain of the Dirichlet form $\cE$ will be denoted by $\cF$.

\begin{theorem}\label{thm:main lifting}
  Assume that $M$ has Ricci curvature bounded below by $K\in\R$. Then
  the following hold:
\begin{enumerate}
\item Bochner inequality: For all cylinder functions $F:\cs\to\R$ we have
   \begin{align*}
    \frac12\Delta^\cs\abs{\nabla^\cs F} -\ip{\nabla^\cs F,\nabla^\cs\Delta^\cs F} ~\geq~ K\ \abs{\nabla^\cs F}^2\;.    
  \end{align*}
\item Gradient estimate on $\cs$: For any function $F\in \cF$ we have
  \begin{align*}
    \Gamma^\cs (T^\cs_t F) ~\leq~ \e^{-2Kt} T^\cs_t \Gamma^\cs( F) \quad \pi-a.e.
\end{align*}
\item Wasserstein contraction: For all $\mu,\nu\ll \pi$ we
  have:
  \begin{align*}
    W_{2,\dcs}(H^\cs_t\mu,H^\cs_t\nu)~\leq~ \e^{-Kt}W_{2,\dcs}(\mu,\nu)\;.
  \end{align*}
\end{enumerate}
\end{theorem}

The Bochner inequality, the gradient estimate and the Wasserstein
contraction will be derived by a suitable ``lifting'' of the
corresponding statements on $M$. For the latter two this relies on a
representation of the heat semigroup $T_t^\cs$ as an infinite product
of independent copies of the heat semigroup on $M$ that will be
established in Theorem \ref{thm:sg-welldef}. To our knowledge this
identification is new in the present generality assuming a (possibly
negative) Ricci bound.

Somehow surprisingly, there does not seem to be a straightforward way
to ``lift'' the EVI or the convexity of the relative entropy from $M$
to $\cs.$ Nevertheless, using a careful approximation procedure it is
possible to adapt the techniques of \cite{AGS12} to the setting of the
configuration space to derive it from the gradient estimate
established in Theorem \ref{thm:main lifting}.

\begin{theorem}\label{thm:main EVI}
  Assume that $M$ has Ricci curvature bounded below by $K\in\R$. Then
  the heat flow is the gradient flow of the entropy in the sense of
  the $EVI_K$: For all $\mu,\sigma\in\cP(\cs)$ with
  $\ent(\sigma)<\infty$ and $W_{2,\dcs}(\mu,\sigma)<\infty$ and a.e. $t>0$:
 \begin{align*}
    \ddt \frac12 W_{2,\dcs}^2(H^\cs_t\mu,\sigma) + \frac{K}2 W_{2,\dcs}^2(H^\cs_t\mu,\sigma)~\leq~\ent(\sigma|\pi)-\ent(H^\cs_t\mu|\pi)\;.
  \end{align*}
\end{theorem}
 
Note that a priori the dual semigroup is only defined on measures with
density. Using a careful approximation technique given in Lemma
\ref{lem:ac-approx} and Wasserstein contractivity we can extend it to
all measures at finite distance to the domain of the entropy. This is
also the maximal set of measures for which EVI can be stated. As a
direct consequence we obtain

\begin{corollary}
  The entropy is (strongly) $K-$convex on
  $\big(\cP(\cs),W_{2,\dcs}\big)$. More precisely, for all
  $\mu_0,\mu_1\in D(\ent)$ with $W_{2,\dcs}(\mu_0,\mu_1)<\infty$ and any
  geodesic $(\mu_s)_{s\in[0,1]}$ connecting them we have for all
  $s\in[0,1]$:
  \begin{align*}
    \ent(\mu_s|\pi)~\leq~(1-s)\ent(\mu_0|\pi) + s \ent(\mu_1|\pi) - \frac{K}{2}s(1-s)W^2_{2,\dcs}(\mu_0,\mu_1)\;.
  \end{align*}
\end{corollary}

In particular, we see that the $(\cs,d_\cs,\pi)$ is an extended metric
measure space satisfying the synthetic Ricci bound CD$(K,\infty)$ in
the sense of {\sc Sturm} and {\sc Lott--Villani}.

Since the configuration space naturally appears (see e.g.\
\cite{AKR98b, Os12, Os13}) as the state space for infinite systems of
interacting Brownian motions, our results can be interpreted as a
first step in order to make tools from optimal transportation
available for infinite particle systems. In fact, for the case of no
interaction Theorem \ref{thm:main EVI} is the realization of the
famous heat flow interpretation of {\sc Jordan--Kinderlehrer--Otto}
for an infinite system of Brownian motions. It is a challenge for
future work to incorporate interactions in this picture.

\begin{remark}\label{rem:weighted case}
  It would be natural to consider more generally as base space a
  \emph{weighted} Riemannian manifold $(M,d,\e^{-V}m)$, with
  $V:M\to\R$ say of class $C^2$, and equip the configuration space
  $(\cs,d_\cs)$ with the Poisson measure $\pi_V$ built from the
  reference measure $\e^{-V}m$. This corresponds to a system of
  independent Brownian motions with drift. We expect that all the
  results presented here continue to hold under the assumption of a
  lower bound of the weighted Ricci curvature
  \begin{align*}
    \Ric +\Hess V~\geq~ K\;.
  \end{align*}
  The only thing that does not adapt immediately is the control on the
  tail of the heat kernel in Lemma \ref{lem:HK-tail} needed for the
  explicit representation of the heat semigroup. In fact, the validity
  of such a heat kernel bound under weighted Ricci bounds is
  interesting in itself and seems to be open in this generality. Since
  settling this question is not in the scope of this paper we chose to
  work with unweighted manifolds.
\end{remark}

\subsection*{Connection to the literature}
Even though the article \cite{AKR98} triggered off an enormous amount
of research, the curvature of the ``lifted'' geometry on the
configuration space has - to our knowledge - not yet been
analyzed. {\sc Privault} \cite{Pri01} derived a Weitzenb\"ock type
formula on the configuration space; however, his analysis is based on
a different geometry which does not directly relate to the geometry
introduced in \cite{AKR98}.

Spaces satisfying (synthetic) lower Ricci curvature bounds are
currently a hot topic of research and many impressive results have
been obtained, e.g.\ see \cite{AGS11a, AGS11b, AGS12, EKS}. However,
most of the applications and examples are finite dimensional. So far
the Wiener space was the only known example of a truly infinite
dimensional $\mathsf{CD}$ space. Recently, also path spaces over a Riemannian
manifold have been investigated by {\sc Naber} \cite{Na13} where he
characterizes simultaneous lower and upper Ricci curvature bounds via
gradient estimates and spectral gap estimates on the path space.

The geometry on the configuration space is very similar to the
geometry of the Wasserstein space. However, due to the fact that every
point in a configuration gets mass at least one the lower sectional
curvature bound is stable even for negative lower bounds in contrast
to the Wasserstein space, see Proposition 2.10 in
\cite{sturm:gmms1}. Moreover, the Wasserstein space together with the
entropic measure is known to not admit any Ricci lower bounds
\cite{Cho12} which is again in sharp contrast to Theorem \ref{thm:main
  lifting} and Theorem \ref{thm:main EVI}.

\subsection*{Outline}

In Section \ref{sec:prelim} we start by explaining the ``lifted''
geometry on $\cs$. Using a version of Rademachers Theorem on the
configuration space we show that differential structure and the metric
structure fit together by proving that the Cheeger energy and the
Dirichlet form coincide. Subsequently, we discuss the heat semigroup
in some detail and give a useful point wise representation in terms of
the semigroup on the base space $M$. We close this section by
collecting some tools we need in the proof of the main theorems.

In Section \ref{sec:OT}, we collect and adapt results on optimal
transport to the configuration space setting.

In Section \ref{sec:curv-manifest}, we prove Theorem \ref{thm:main
  lifting}, the different manifestations of curvature bounds which can
be deduced by ``lifting'' of the corresponding results on $M$.

Finally in Section \ref{sec:RCD}, we show that the $EVI_K$ holds on
the configuration space, i.e.\ we prove Theorem \ref{thm:main EVI}.

The Appendix contains the proof of the approximation result needed to
extend the dual semigroup beyond measures with density.

\subsection*{Acknowledgements}

The authors would like to thank Theo Sturm and Fabio Cavalletti for several fruitful discussions on the subject of this paper.

\section{Preliminaries}
\label{sec:prelim}

\subsection{Differentiable structure of configuration space}
\label{sec:diff}

Let $M$ be a smooth complete and connected Riemannian manifold. We
denote by $\ip{\cdot,\cdot}_x$ the metric tensor at $x$, $d$ is the
Riemannian distance and $m$ the volume measure. We assume that $M$ is
non-compact and $m(M)=\infty$.\footnote{The results also hold in the case that $M$ is compact. However, they can be derived much easier.} The configuration space $\cs$ over the
base space $(M,d,m)$ is the set of all locally finite counting
measures, i.e.\
\begin{align*}
  \cs:=\{\gamma\in\mathcal M(M): \gamma(K)\in\N_0\ \text{ for all }
  K\subset M \text{ compact }\}\;.
\end{align*}
Each $\gamma\in\cs$ can be represented as $\gamma=\sum_{i=1}^n
\delta_{x_i}$ for some $n\in\N_0\cup \{\infty\},$ and suitable points
$x_i$ in $M$. Here $n=0$ corresponds to the empty configuration. To be
more precise, let $\cA$ be the set of finite and infinite sequences in
$M$ without accumulation points and let
\begin{align*}
  l:\cA\to\cs,\ (x_1,x_2,,\ldots)=\mathbf{x}\mapsto \gamma=\sum_{i} \delta_{x_i}\;.
\end{align*}
Then any $\mathbf{x}\in l^{-1}(\gamma)$ is called a labeling of
$\gamma$. We can decompose the configuration space as $\cs=\cup_{n\in
  \N_0\cup\{\infty\}} \cs^{(n)}$ where $\cs^{(n)}=\{\gamma\in\cs:
\gamma(M)=n\}.$

We endow the configuration space with the vague topology which makes
it a Polish space as a closed subset of a Polish space (e.g.\ see
\cite[Theorem A2.3]{kallenberg2002foundations}). This means that
$\gamma_n\to\gamma$ if and only if $\int f\ d\gamma_n\to\int f\
d\gamma =:\gamma(f)$ for all $f\in C_c(M).$

There is a natural probability measure on $\cs$, the Poisson measure
$\pi$. It can be defined via its Laplace transform
$$ \int \exp(\gamma(f))\ \dd\pi(\gamma) = \exp\left(\int \exp(f(x))-1\ \dd m(x)\right).$$
Equivalently, we can characterize $\pi$ as follows: for any choice of
disjoint Borel sets $A_1,\ldots ,A_k\subset M$ with $m(A_i)<\infty$
the family of random variables $\gamma(A_1),\ldots,\gamma(A_k)$ is
independent and $\gamma(A_i)$ is Poisson distributed with parameter
$m(A_i)$. In particular, given a Borel set $A$ of finite volume and
condition on the event that $\gamma(A)=n<\infty$ then the $n$ points
are iid uniformly distributed in $A$.

Note that the analysis of $\cs$ is most interesting when $M$ in
non-compact and $m(M)=\infty$ since in this case configurations consist
typically of infinitely many points, i.e.\ we have $\pi(\cs^{(n)})=0$
for all $n\in\N$ and $\pi(\cs^{(\infty)})=1$.

The \emph{tangent space} $T_\gamma\cs$ of $\cs$ at a configuration
$\gamma$ is defined to be the space of all $\gamma$-square integrable
sections of the tangent bundle $TM$ of $M$, i.e.\
$$T_\gamma\cs=\{V:M\to TM, \int_M \langle V,V\rangle_x\ \dd\gamma(x)<\infty\}.$$
Equivalently, we can write $T_\gamma\cs=L^2(\bigoplus_{x\in\gamma} T_xM, \gamma).$
We will denote the scalar product on $T_\gamma\cs$ by 
$$ \langle V_1,V_2 \rangle_\gamma := \int_M \langle V_1(x),V_2(x)\rangle_x\ \dd\gamma(x).$$
We also sometimes write $\|V\|^2_\gamma:=\langle V, V\rangle_\gamma.$
Note that this is a non-trivial structure. The tangent spaces vary
with $\gamma$ even if $M$ is Euclidean.

Next we introduce an important class of ``test functions''. A smooth
\emph{cylinder function} is a function $F:\cs\to\R$ that can be
written as
\begin{align*}
  F(\gamma)=g_F(\gamma(\phi_1),\ldots,\gamma(\phi_n))\;,
\end{align*}
for some $n\in\N, g_F\in C_b^\infty(\R^n)$ and
$\phi_1,\ldots,\phi_n\in C_c^\infty(M).$ The set of all smooth
cylinder functions will be denoted by $\mathsf{Cyl}^\infty(\cs).$ For
$F\in \mathsf{Cyl}^\infty(\cs)$ we define the \emph{gradient} of $F$ by
\begin{align*}
  \nabla^\cs F(\gamma;x):= \sum_{i=1}^n \partial_i
  g_F(\gamma(\phi_1),\ldots,\gamma(\phi_n)) \nabla \phi_i(x) \quad
  \gamma\in \cs,\ x\in M\;.
\end{align*}
Here $\partial_i$ denotes the partial derivative in the i-th direction
and $\nabla$ denotes the gradient on $M$. Alternatively, we can define
the gradient using directional derivatives. To this end denote the set
of all smooth and compactly supported vector fields on $M$ by
$\mathcal V_0(M)$. For $V\in\mathcal V_0(M)$ let $\psi_t$ be the flow
of diffeomorphisms generated by $V$. For fixed $\gamma\in\cs$, this
generates a curve $\psi_t^*\gamma=\gamma\circ \psi_t^{-1}, t\in\R$ on
$\cs$. Then we have for $F\in\cyl$ and $\gamma\in\cs$
\begin{align*}
  \left. \frac{\dd}{\dd t}\right|_{t=0} F(\psi_t^*\gamma)=\langle
  \nabla^\cs F(\gamma),V\rangle_\gamma =:\nabla^\cs_VF(\gamma)\;.
\end{align*}
Similarly, we can introduce the \emph{divergence} $\mbox{div}^\cs$ on
$\cs$. For $F_i\in \cyl$ and $V_i\in \mathcal V_0(M)$ we define for
$\gamma\in \cs$
\begin{align*}
  \mbox{div}^\cs\left(\sum_{i=1}^n F_i\cdot V_i\right)(\gamma):=
  \sum_{i=1}^n \nabla^\cs_{V_i}F_i(\gamma) + F_i(\gamma)\cdot
  \gamma(\mbox{div}^M(V_i))\;.
\end{align*}
It is proven in \cite{AKR98} that the Poisson measure $\pi$ is (up to
the intesity) the unique measure such that $\mbox{div}^\cs$ and
$\nabla^\cs$ are adjoint in $L^2(\pi)$. We also define the
\emph{Laplace operator} $\Delta^\cs:= \mbox{div}^\cs\nabla^\cs$.

With this differential structure at hand we can talk about
\emph{Dirichlet forms}. For a cylinder function $F\in \cyl$ we define
the pre-Dirichlet form
\begin{align*}
  \cE(F,F)\ :=\ \int \langle \nabla^\cs F,\nabla^\cs F\rangle_\gamma\
  \pi(d\gamma)\;.
\end{align*}
It is shown in \cite{AKR98} that $(\cE,\mathsf{Cyl}^\infty(\cs))$ is
closable and its closure $(\cE,\mathcal F)$ is a Dirichlet form. By
\cite[Proposition 1.4 (iv)]{RS99}, for every $F\in\cF$ there exists a
measurable section $\nabla^\cs F:\cs\to T\cs$ such that
$\cE(F)=\int\abs{\nabla^\cs F}_\gamma^2\dd\pi(\gamma)$. Thus $\cE$
admits a carr\'e du champs operator $\Gamma^\cs:\cF\to L^1(\cs,\pi)$ given
by $\Gamma^\cs(F)(\gamma)=\abs{\nabla^\cs F}^2_\gamma$.  

We will denote the semigroup in $L^2(\cs,\pi)$ associated to the
Dirichlet form $(\cE,\cF)$ by $T^\cs_t=\exp(t\Delta^\cs)$ and call it the
\emph{heat semigroup} on $\cs$. Its generator is the Friedrichs
extension of $\Delta^\cs$.

\subsection{Metric structure of $\cs$ and compatibility}
\label{sec:coincidence-diff-metric}

A natural distance on the configuration space is given by the
non-normalized $L^2$-transportation distance, defined for two measures
$\gamma,\eta\in\cs$ by
\begin{align*}
  \dcs^2(\gamma,\eta)=\inf_{q\in\mathsf{Cpl}(\gamma,\eta)} \int
  d^2(x,y)\ q(dx,dy)\;,
\end{align*}
where $\mathsf{Cpl}(\gamma,\eta)$ denotes the set of all couplings
between $\gamma$ and $\eta$. Note that
$\dcs:\cs\time\cs\to[0,+\infty]$ is an \emph{extended} distance,
i.e.\ it is symmetric, vanishes precisely on the diagonal and satisfies
the triangle inequality. It can take the value $+\infty$, e.g. we have
$\dcs(\gamma,\eta)=\infty$ if $\gamma\in\cs^{(n)}$ and
$\eta\in\cs^{(m)}$ with $m\neq n.$

We denote by $C(\cs)$ the set of all continuous functions on $\cs$
w.r.t.\ the vague topology. We say that a function $F:\cs\to\R$ is
$\dcs$-Lipschitz iff
\begin{align}\label{eq:lip}
  \abs{F(\gamma)-F(\eta)}\leq C \dcs(\gamma,\eta)\quad\forall \gamma,\eta\in\cs\;,
\end{align}
for some constant $C\geq0$. The set of all $\dcs$-Lipschitz functions
will be denoted by $\Lip(\cs)$ and the set of bounded $\dcs$-Lipschitz
functions by $\Lip_b(\cs).$ For $F\in\Lip(\cs)$ the global Lipschitz
constant $\Lip(F)$ is the smallest $C$ such that \eqref{eq:lip} holds
and we define the local Lipschitz constant by
\begin{align}\label{def:locLip}
|DF|(\gamma):=\limsup_{\dcs(\eta,\gamma)\to 0} \frac{|F(\gamma)-F(\eta)|}{\dcs(\gamma,\eta)}\;.
\end{align}

The compatibility of the differential and metric structure of the
configuration space is given by the following Rademacher theorem which
we quote from \cite[Thm. 1.3, Thm. 1.5]{RS99}.

\begin{theorem}\label{thm:Rademacher}
 \begin{enumerate}
 \item Suppose $F\in L^2(\pi)\cap\mathsf{Lip}(\cs)$. Then
   $F\in\cF$. Moreover, there exists a measurable section $\nabla^\cs
   F$ of $T\cs$ such that
\begin{enumerate}
\item[a)] $\Gamma^\cs(F)(\gamma)=\|\nabla^\cs F(\gamma)\|_\gamma \leq \Lip(F)$ for $\pi$-a.e. $\gamma$.
\item[b)] If $V\in\mathcal V_0(M)$ generates the flow $(\psi_t)_{t\in\R}$, then for $\pi$-a.e. $\gamma$ and all $s\in\R$:
  \begin{align*}
    \frac{F(\psi_t^*\gamma)-F(\gamma)}{t}\to \langle \nabla^\cs
    F(\gamma),V\rangle_\gamma,\quad \text{ as } t\to 0 \text{ in } L^2(\pi\circ(\psi_s^*)^{-1})\;.
  \end{align*}
\end{enumerate}
\item If $F\in\mathcal F$ satisfies $\Gamma^\cs(F)\leq C^2, \pi$-a.e. and
  if $F$ has a $\dcs$-continuous $\pi$-version, then there exists a
  $\pi$-measurable $\pi$-version $\tilde F$ which is $\dcs$-Lipschitz
  with $\mbox{Lip}(\tilde F)\leq C.$
\item $\dcs$ coincides with the intrinsic metric of the Dirichlet form $(\cE,\cF)$, i.e.\
     for all $\gamma,\eta\in\cs$:
     \begin{align*}
       \dcs(\gamma,\eta)=\sup\left\{F(\gamma)-F(\eta) : F\in\cF\cap C(\cs)\;,\ \Gamma^\cs(F)\leq1\ \pi\text{-a.e.}\right\}\;.
     \end{align*}
 \end{enumerate}
\end{theorem}

As a consequence we obtain the following pointwise comparison of the
Lipschitz constant and the Gamma operator.

\begin{lemma}\label{lem:lip-gamma}
  For all $F\in\mathsf{Lip}_b(\cs)$ and $\pi$-a.e. $\gamma$ we have
  \begin{align}\label{eq:lip-gamma}
    |DF|(\gamma)\geq\|\nabla^\cs
    F\|_\gamma=\sqrt{\Gamma^\cs(F)}(\gamma)\;.
  \end{align}
\end{lemma}

\begin{proof}
  By \cite[Prop. 5.4]{RS99} for every $\gamma,\eta\in\cs$ with
  $\dcs(\gamma,\eta)<\infty$ and every $\epsilon>0$ there is a
  $V\in\mathcal V_0(M)$ generating the flow $(\psi_t)_{t\in\R}$ such
  that $\dcs(\psi_1^*\gamma,\eta)<\epsilon$ and
  $\|V\|_{\psi^*_t\gamma}=\dcs(\psi_1^*\gamma,\gamma)$ for all
  $t\in[0,1].$ Hence, by $\dcs$-continuity of $F$ we have
  \begin{align*}
    |DF|(\gamma)=\limsup_{\dcs(\eta,\gamma)\to
      0}\frac{|F(\eta)-F(\gamma)|}{\dcs(\eta,\gamma)}=\limsup_{V\in\mathcal
      V_0(M), \|V\|_\gamma\to
      0}\frac{|F(\psi_1^*\gamma)-F(\gamma)|}{\norm{V}_\gamma}\;.
  \end{align*}
  By part (i) b) of Theorem \ref{thm:Rademacher}, we have for
  $\pi$-a.e.  $\gamma$ and all $V\in\mathcal V_0(M)$
  \begin{align*}
    |DF(\gamma)|\geq \lim_{t\to 0}
    \frac{F(\psi_t^*\gamma)-F(\gamma)}{t \|V\|_\gamma} =
    \frac1{\|V\|_\gamma}\langle \nabla^\cs F(\gamma),V\rangle_\gamma\;.
  \end{align*}
  Hence, taking the supremum over $V$ we get $|DF|(\gamma)\geq \|\nabla^\cs
  F\|_\gamma$ for $\pi$-a.e. $\gamma$.
\end{proof}

In \cite{AGS11a} Ambrosio, Gigli and Savar\'e develop a calculus on
(extended) metric measure spaces and study the ``heat flow'' in this
setting. A crucial result is the construction of a natural candidate
for a Dirichlet form starting only from a metric and a measure. Their work
is the foundation for studying Riemannian Ricci curvature bounds via
optimal transport on (non-extended) metric measure spaces in
\cite{AGS11b,AGS12}. Here we make the connection to this approach,
showing that the triple $(\cs,\dcs,\pi)$ fits into the framework of
\cite{AGS11a} and that the Dirichlet form $\cE$ coincides with its
metric counterpart constructed from $\dcs$.

First note that $(\cs,\dcs)$ equipped with the vague topology is a
Polish extended space in the sense of \cite[Def. 2.3]{AGS11a}: it is
complete, i.e.\ every $\dcs$-convergent sequence has a limit in $\cs$,
$\dcs(\gamma_n,\gamma)\to0$ implies that $\gamma_n\to\gamma$ vaguely
for all sequences $(\gamma_n)\subset \cs$ and $\gamma\in\cs$, and
$\dcs$ is lower semi continuous w.r.t.\ the vague topology.
 
The Cheeger energy $\Ch$ introduced in \cite{AGS11a} is given on the
configuration space as a functional $\Ch:L^2(\cs,\pi)\to[0,+\infty]$
defined via

\begin{align}\label{eq:Cheeger energy}
  \mathsf{Ch}(F):=\inf\left\{\liminf_{n\to\infty}
    \frac12\int_\cs|DF_n|^2\ d\pi\ :\ F_n\in\mathsf{Lip}_b(\cs), \
    F_n\to F \text{ in } L^2(\cs,\pi)\right\}\;.
\end{align}

\begin{proposition}\label{prop:E=2Ch}
  The Cheeger energy associated to $\dcs$ coincides with the Dirichlet
  form $\cE$, i.e.\ $\cE(F)=2\Ch(F)$ for all $F\in L^2(\cs,\pi)$.
\end{proposition}

\begin{proof}
  Let us first show that $\cE\leq 2\Ch$. By definition for $F\in
  L^2(\cs,\pi)$ with $\Ch(F)<\infty$ there is a sequence of bounded
  Lipschitz functions $(F_n)_{n\in\N}$ such that $F_n\to F$ in
  $L^2(\cs,\pi)$ and $\lim_n \Ch(F_n)=\Ch(F)$. By \eqref{eq:lip-gamma}
  of Lemma \ref{lem:lip-gamma} and lower semicontinuity of $\cE$ in
  $L^2(\cs,\pi)$ we obtain
  \begin{align*}
    2\Ch(F)=\lim_n\int |DF_n|^2\ d\pi \geq \liminf_n \int |\nabla^\cs
    F_n|^2\ d\pi \geq \cE(F)\;.
  \end{align*}
  To prove the converse inequality $\cE\geq 2\Ch$, note that by
  definition for $F\in L^2(\cs,\pi)$ with $\cE(F)<\infty$ there exists
  a sequence of cylinder functions $(F_n)_{n\in\N}$ such that $F_n\to
  F$ in $L^2(\cs,\pi)$ and $\lim_n\cE(F_n)=\cE(F)$. Note that any
  cylinder function $F_n$ is $\dcs$-Lipschitz with $\abs{D
    F_n}(\gamma) = \norm{\nabla^\cs F_n}_\gamma$. Thus we obtain from
  the definition of $\Ch$ and its lower semicontinuity in
  $L^2(\cs,\pi)$ (see \cite[Thm. 4.5]{AGS11a}):
  \begin{align*}
    2\Ch(F)\leq \liminf_n 2\Ch(F_n)\leq \liminf_n \int |DF_n|^2\ d\pi = \liminf_n\cE(F_n)=\cE(F)\;.
  \end{align*}
\end{proof}

Having identified the Dirichlet form $\cE$ with the Cheeger $\Ch$
energy build from the distance $d_\cs$ in particular yields that the
semigroup $T^\cs_t$ coincides with the gradient flow of $\Ch$ in
$L^2(\cs,\pi)$. This will be used in Section \ref{sec:RCD}.

\subsection{The heat semigroup}
\label{sec:heat-sg}

In this section we establish an explicit representation of the Markov
semigroup $T^\cs_t$ associated to the Dirichlet form $\cE$. We
identify it with the semigroup of the independent particle process
obtained by starting in each point of a configuration independent
Brownian motions. This identification is non-trivial when
$m(M)=\infty$. While the first lives by definition on the
configuration space, the latter a priory lives in the larger space of
counting measures that are not necessarily locally finite. We will
show that whenever $\Ric_M\geq K$ for some $K\in \R$ the independent
particle process can be started in a subset of $\cs$ of full $\pi$
measure and stays there for all time. 

Consider the infinite product $M^\N$ equipped with the cylinder
$\sigma$-algebra $\cC(M^\N)$. We put $\cA\in \cC(M^\N)$ to be the set
of all sequences $(x_n)_{n=1}^\infty\in M^\N$ which have no
accumulation points. Recall the labeling map $l:\cA\to \cs$ given by
\begin{align*}
  l:\ (x_n)_{n=1}^\infty ~\mapsto \sum\limits_{n=1}^\infty\delta_{x_n}\;.
\end{align*}
Note that $\pi(l(\cA))=1$. Let $p^M_t(x,y)$ denote the heat kernel on
the manifold $M$. Moreover, we denote by
\begin{align*}
  p^M_t(x,A)~=~\int\limits_A p^M_t(x,y)\dd m(y)
\end{align*}
the semigroup of transition kernels. This gives rise to a family of
probability measure on $\big(M^\N,\cC(M^\N)\big)$ by considering the
product measures
\begin{align*}
  p_t^\N\big((x_n)_n,\cdot\big) := \bigotimes\limits_{n=1}^\infty
  p^M_t(x_n,\cdot)\;.
\end{align*}
Given $\gamma\in\cs$ we can define a probability measure on $\cs$ via
\begin{align}\label{eq:naivsgcs}
  p^\cs_t(\gamma,G)~:=~ p_t^\N\big((x_n)_n,l^{-1}(G)\big)\qquad G\in \cB(\cs)\;,
\end{align}
where $\gamma=l\big((x_n)\big)$, provided that
$p_t^\N\big((x_n)_n,\cA\big)=1$ for all $t\geq 0$. Our goal will be to
show that for a large class of $\gamma$ the latter indeed holds.

We fix a point $x_0\in M$ and denote by $B_r=B(x_0,r)$ the closed ball
around $x_0$ with radius $r$. Define for each $\a\geq1$:
\begin{align*}
  \Theta_\a := \left\{\gamma\in\cs\ :\ \exists C>0:\ \forall r\in\N: \gamma\big(B_r\big)\leq C \e^{\a r}\right\}\;.
\end{align*}
Since $\Theta_\a\subset\Theta_\beta$ for $\a\leq\beta$ it makes sense
to define also
\begin{align}\label{good config}
  \Theta := \bigcup\limits_{\a\geq1}\Theta_\a\;.
\end{align}
We call $\Theta$ the set of \emph{good configurations}. Note that
the Poisson measure is concentrated on configurations satisfying
$\gamma(B_r)\sim\vol(B_r)$ as $r\to\infty$. Since we assume $\Ric\geq
K$, the Bishop--Gromov volume comparison theorem (see Lemma
\ref{lem:HK-tail} below) implies that $\vol(B_r)\leq C\e^{\a r}$ for
suitable constants $C,\a$. Thus, we conclude that $\pi(\Theta_\a)=1$
for $\a$ sufficiently large and in particular $\pi(\Theta)=1$. The
following is a slight generalization of \cite[Thm. 2.2, 4.1]{KLR08}.

\begin{theorem}\label{thm:sg-welldef}
  Assume that $\Ric_M\geq K$ for some $K\in\R$. Then for each
  $\gamma\in\Theta$ and all $t>0$ the measure $p^\cs_t(\gamma,\cdot)$
  defined in \eqref{eq:naivsgcs} is a probability measure on
  $\cs$. Moreover, $(p^\cs_t)_{t\geq0}$ is a Markov semigroup of kernels
  on $\big(\Theta,\cB(\Theta)\big)$. For each $F\in L^2(\cs,\pi)$ the
  function
  \begin{align*}
    \Theta\ni\gamma\mapsto \tilde T_t^\cs F(\gamma)=\int_\Theta F(\xi)p^\cs_t(\gamma,\dd\xi)
  \end{align*}
  is a $\pi$-version of the function $T^\cs_tF\in L^2(\cs,\pi)$.
\end{theorem}

\begin{proof}
  Let us write $\abs{x}:=d(x,x_0)$, where $x_0$ is the point chosen in
  the definition of $\Theta$.  We will first prove that for any
  $\gamma\in\Theta$ and $t\in(0,\eps)$:
  \begin{align}\label{eq:12ball}
    \sum\limits_{x\in\gamma} p^M_t\big(x,\complement B(x,\abs{x}/2)\big) ~<~ \infty\;.
  \end{align}
  To this end let $\gamma\in\Theta$ and let $(x_n)_n$ be a labeling of
  $\gamma$. We can assume that $\abs{x_n}\leq\abs{x_{n+1}}$ for all
  $n$. There exists $C,\a$ such that $\gamma(B_r)\leq C \e^{\a r}$ for
  all $r\in\N$. For $n\in\N$ let us set:
  \begin{align*}
    r_n~:=~\left\lfloor\frac{1}{\a}\log\Big(\frac{n}{C}\Big)\right\rfloor\;.
  \end{align*}
  This implies that $\gamma(B_{r_n})<n$ and hence we have $x_n\notin
  B_{r_n}$ and $\abs{x_n}>r_n$. Using Lemma \ref{lem:HK-tail} below we
  obtain that for constants $C_1,C_2$ (possibly changing from line to
  line):
  \begin{align*}
    &\sum\limits_{x\in\gamma} p^M_t\big(x,\complement B(x,\abs{x}/2)\big)
    ~=~
    \sum\limits_{n=1}^\infty p^M_t\big(x_n,\complement B(x_n,\abs{x_n}/2)\big)
    ~\leq~
    \sum\limits_{n=1}^\infty p^M_t\big(x_n,\complement B(x_n,r_n/2)\big)\\
    ~&\leq~
    \sum\limits_{n=1}^\infty C_2\exp({-C_1r_n^2})
    ~\leq~
    \sum\limits_{n=1}^\infty C_2 \exp\Big(- C_1\log(n)^2 \Big)
    ~<~\infty\;,
  \end{align*}
  which proves \eqref{eq:12ball}.

  Now, we want to prove that for any $(x_n)_n\in l^{-1}(\Theta)$ we
  have
  \begin{align}\label{eq:mass1}
    p^\N_t\Big((x_n)_n,l^{-1}(\Theta)\Big)~=~1\;.
  \end{align}
  So fix $(x_n)_n\in l^{-1}(\Theta)$ and set
  \begin{align*}
    \cA_n~&:=~\left\{(y_k)_k\in M^\N\ :\ y_n\in B(x_n,\abs{x_n}/2) \right\}\;,\\
    \cA'~&:=~\liminf\limits_n\cA_n\;.
  \end{align*}
  From \eqref{eq:12ball} and the Borel--Cantelli lemma we infer that
  for any $t\in(0,\eps)$:
  \begin{align*}
    p_t^\N\Big((x_n)_n, \cA'\Big)~=~1\;.
  \end{align*}
  By definition of $\Theta$ we have $\abs{x_n}\to\infty$ as
  $n\to\infty$ and so no sequence in $\cA'$ has accumulation points
  which means $\cA'\subset\cA$. To prove \eqref{eq:mass1} it is
  sufficient to show that $\cA'\subset l^{-1}(\Theta)$. So fix
  $(y_n)_n\in\cA'$ and let $k$ be the number of those $n$ such that
  $y_n\notin B(x_n,\abs{x_n}/2)$. Putting $\gamma=l\big((x_n)_n\big)$
  and $\gamma'=l\big((y_n)_n\big)$ and using \eqref{eq:volgrowth} we
  can estimate:
  \begin{align*}
    \gamma'(B_r)~&\leq~\gamma(B_{2r}) + k\\
                &\leq~C \e^{2 \a r} + k\\
                &\leq~C' \e^{2\a r}\;, 
  \end{align*}
  for a suitable $C'>0$ and all $r\in\N$. Hence we have
  $\gamma'\in\Theta_{2\a}\subset\Theta$ and this proves
  \eqref{eq:mass1}. Thus \eqref{eq:naivsgcs} defines a probability
  measure on $\cs$ concentrated on $\Theta$. It then follows easily
  from the semigroup property of $p^\N_t$ that $p^\cs_t$ can be defined
  for all $t>0$ and is a Markov semigroup of kernels on $\Theta$. The
  last statement of the theorem is proven as in
  \cite[Thm. 2.1]{KLR08}.
\end{proof}

\begin{lemma}\label{lem:HK-tail}
  Assume that $\Ric_M\geq -K$ for some $K\in[0,\infty)$. Then there is
  a constant $c$ such that
  \begin{align}\label{eq:volgrowth}
   \vol \big(B(x,r)\big)~\leq~\vol\big(B(x,1)\big)\cdot \e^{c r}\qquad \forall x\in M,\; r\geq1\;. 
  \end{align}
  Moreover, for any $T>0$ there are constants $c_1,c_2$ such that:
  \begin{align}\label{eq:HK-tail}
    \sup\limits_{t\in(0,T]}\sup\limits_{x\in M} p^M_t\big(x,\complement B(x,r)\big) ~\leq~ c_2 \e^{-c_2 r^2}\qquad \forall r>0\;.
  \end{align}
\end{lemma}

\begin{proof}
  The estimate \eqref{eq:volgrowth} follows from the Bishop--Gromov
  volume comparison theorem \cite[Lem.\ 5.3.bis]{Gr:metricstructure}.
 
  The second estimate \eqref{eq:HK-tail} is a consequence of the
  following result (see relation (8.65) in \cite{Str00}): Fix $x\in M$
  and let $(B^x_t)_{t\geq0}$ be a Brownian motion started from $x$. Then
  for any $\lambda\in(0,1)$ and $r>0$ we have:
  \begin{align}\label{eq:Stroock}
    \P\Big[\sup\limits_{0\leq s\leq t} d(B^x_s,x)\geq r\Big]
     ~\leq~
   \frac{2}{\sqrt{1-\lambda}}\exp\left( -\frac{\lambda r^2}{2t} + \frac{\lambda\big(2d+Kd^2t\big)}{1-\lambda}\right)\;,    
  \end{align}
  where $d=\dim M$. This implies \eqref{eq:HK-tail} immediately, since
  \begin{align*}
    p^M_t\Big(x,\complement B(x,r) \Big)~=~\P\Big[d(B^x_t,x)\geq r\Big]~\leq~\P\Big[\sup\limits_{0\leq s\leq t} d(B^x_s,x)\geq r\Big]\;.
  \end{align*}
\end{proof}

\subsection{Additional tools}
\begin{lemma}\label{lem:existence of matching}
  For every $\gamma, \omega \in \cs$ with $\dcs(\gamma,\omega)<\infty$
  there exists an optimal coupling $q$ which is a matching,
  i.e.\ $\dcs^2(\gamma,\omega)=\int d^2(x,y)\dd q(x,y)$ and for all
  $\{x,y\}\in M\times M$ we have $q(\{x,y\})\in \{0,1 \}$.
\end{lemma}

As an immediate consequence we obtain that
\begin{align*}
  \dcs^2(\gamma,\omega)~=~\min\left\{\sum_{i=1}^nd^2(x_i,y_i)\ :\ \gamma=\sum_i\delta_{x_i}, \omega=\sum_i\delta_{y_i}\right\}\;,
\end{align*}
provided $\dcs(\gamma,\omega)<\infty$ and $\gamma(M)=\omega(M)=n$.

\begin{proof}
  By \cite{Kendall:infiniteBirkhoff}, the set of doubly stochastic
  matrices is the closure of convex combinations of permutations
  matrices, i.e.\ doubly stochastic matrices whose entries are
  precisely 0 or 1, with respect to the locally convex topology which
  makes all elements, row sums and column sums of the matrix
  continuous. Call this the $\tau$ topology. Now take
  $q\in\mathsf{Cpl}(\gamma,\omega)$ and $f\in C_c(M\times M).$ Then
  $\int f \dd q= \sum f(x_i, y_j) q(x_i,y_j),$ for some labeling
  $(x_i)_i, (y_j)_j$ of $\gamma$ and $\omega$ respectively. Then
  $(a_{ij}=q(x_i,y_j)$ defines a doubly stochastic matrix. Fixing
  the labeling, a doubly stochastic matrix defines a coupling between
  $\gamma$ and $\omega.$ Moreover, as $\sum f(x_i, y_j) a_{ij}$
  is a finite sum, convergence in the $\tau$ topology implies
  convergence in the vague topology.

  Take $q'\in\Opt(\gamma,\omega)$. By the results of
  \cite{Kendall:infiniteBirkhoff}, there exists a sequence of
  couplings $(q_n')_n$ converging vaguely to $q'$ such that each
  $q_n'$ can be written as a (finite) convex combination of matchings
  (which correspond to permutation matrices). By the linearity of
  $q\mapsto \int d^2 \dd q$, this implies the existence of a sequence
  of matchings $(q_n)_n$ of $\gamma$ and $\omega$ such that $\int d^2
   \dd q_n\leq \int d^2 \dd q_n' \searrow \int d^2 \dd q'$. Hence,
  we have a uniform bound on the transportation cost and there is a
  converging subsequence which we denote again by $(q_n)_n$. Denote by
  $q$ its limit. By lower semicontinuity, we have
  \begin{align*}
    \int d^2 \ \dd q \leq \liminf \int d^2\ \dd q_n = \int d^2\ \dd q'\;,
  \end{align*}
  so that $q\in \Opt(\gamma,\omega).$ As all the $q_n$ are matchings
  also $q$ has to be a matching which can be seen by testing against
  functions $f_{i,j}\in C_c(M\times M)$ which satisfy
  $f_{i,j}(x_l,y_k)=\delta_{x_i,y_j}(x_k,y_k)$ for the fixed labeling
  $(x_i)_i$ and $(y_j)_j$ of $\gamma$ and $\omega$.
\end{proof}

\begin{corollary}\label{cor:geo}
  $(\cs,\dcs)$ is a geodesic space, i.e any pair $\gamma_0,\gamma_1$
  with $\dcs(\gamma_0,\gamma_1)<\infty$ can be connected by a curve
  $(\gamma_t)_{t\in[0,1]}$ such that for all $s,t\in[0,1]$ we have
  $\dcs(\gamma_s,\gamma_t)=\abs{t-s}\dcs(\gamma_0,\gamma_1)$.
\end{corollary}

\begin{proof}
  Choose labelings $(x_i^j)_i$ of $\gamma_j$ such that
  $\dcs^2(\gamma_0,\gamma_1)=\sum_id^2(x_i^0,x_i^1)$. For each $i$
  choose a geodesic $(x_i^t)_{t\in[0,1]}$ and put
  $\gamma_t=\sum_i\delta_{x_i^t}$. Then $(\gamma_t)_t$ is a geodesic
  in $\cs$. Indeed,
  \begin{align*}
    \dcs^2(\gamma_s,\gamma_t)~\leq~ \sum_id^2(x_i^s,x_i^t) ~=~
    \abs{t-s}^2 \sum_i d^2(x_i^0,x_i^1) ~=~
    \abs{t-s}^2\dcs^2(\gamma_0,\gamma_1)\;.
  \end{align*}
  The reverse inequality follows from the triangle inequality.
\end{proof}

\section{Optimal transport on configuration space}\label{sec:OT}

We denote the set of probability measures on $\cs$ by $\cP(\cs).$ For
$\mu,\nu\in\cP(\cs)$ the $L^2$-Wasserstein distance is defined via
\begin{align*}
  W_2^2(\mu,\nu):=\inf_{q\in\Cpl(\mu,\nu)} \int \frac{1}{2}
  \dcs^2(\gamma,\eta)\ q(d\gamma,d\eta)\;,
\end{align*}
where $\Cpl(\mu,\nu)$ denotes the set of all couplings between $\mu$
and $\nu$. A minimizer is called optimal coupling and the set of all
optimal couplings between $\mu$ and $\nu$ will be denoted by
$\Opt(\mu,\nu).$ This transportation problem has been studied in the
case of $M=\R^k$ in \cite{De08}; the generalization to Riemannian
manifolds is straightforward. The main result states

\begin{theorem}[\cite{De08}]
  Let $\mu,\nu\in\cP(\cs)$ with $W_2(\mu,\nu)< \infty.$ Assume that
  $\mu\ll\pi$. Then, there is a unique optimal coupling $q$ which is
  induced by a transportation map, i.e.\ $q=(id,T)_*\mu.$
\end{theorem}

\subsection{Duality and Hopf--Lax semigroup}\label{sec:dual and HL}

By general theory, see \cite[Thm. 2.2]{Kel84}, we have the following
Kantorovich duality

\begin{theorem}\label{thm:duality}
  Let $\mu,\nu\in P(\cs)$ such that $W_2(\mu,\nu)<\infty$. Then we
  have 
  \begin{align}\label{eq:duality1}
    W^2_2(\mu,\nu)~=~\sup\left\{\int\phi^c \dd \nu + \int \phi\dd \mu~:~\phi\in C_b(\cs)\right\}\;,
  \end{align}
where the \emph{c-transform} of $\phi$ is defined by
\begin{align*}
  \phi^c(\gamma)~=~\inf_{\eta\in\cs}\left\{\frac12 \dcs^2(\gamma,\eta) - \phi(\eta)\right\}\;.
\end{align*}
\end{theorem}

It is not known if the supremum is attained or not. For a function
$f:\cs\to\R\cup\{\infty\}$ we define the Hopf--Lax semigroup

\begin{align*}
  Q_t f(\gamma)~=~\inf\limits_{\eta\in\cs} \left\{f(\eta) + \frac{\dcs^2(\eta,\gamma)}{2t}\right\}\;.
\end{align*}

The function $Q_t f$ is non trivial on the set
$$ \mathcal D(f):=\{ \gamma\in\cs: \dcs(\gamma,\omega)<\infty \mbox{ for some $\omega$ with } f(\omega)<\infty\}.$$
For $\gamma\in \mathcal D(f)$ we set 
$$t_*(\gamma):=\sup\{ t>0: Q_tf(\gamma)>-\infty\}$$

with the convention that $t_*(\gamma)=0$ if $Q_tf(\gamma)=-\infty$ for
all $t>0.$ If $f$ is bounded also $Q_tf$ is bounded, even
$\dcs$-Lipschitz (with global Lipschitz bound $\mbox{Lip}(Q_tf)\leq
2\sqrt{\mbox{osc}(f)/t}$ where $\mbox{osc}(f)=\sup f-\inf f$), and
$t_*=\infty$ for all $\gamma.$ Note that if $f$ is $\dcs$-Lipschitz,
so is $Q_tf$ with a priori bound (\cite[Prop.~3.4]{AGS11a})
\begin{align}\label{eq:Lip-HL}
  \abs{DQ_s\phi}~\leq~2\Lip(\phi)\;.
\end{align}
Since $(\cs,\dcs)$ is a length space, this implies $\Lip(Q_s\phi)\leq
2\Lip(\phi)$.  For more details we refer to Section 3 of
\cite{AGS11a}. In particular, if $f\in C_b(\cs)$ then $Q_1(-f)=f^c$ is
$\dcs-$Lipschitz. Hence, we have

\begin{corollary}\label{cor:duality}
  Let $\mu,\nu\in P(\cs)$ such that $W_2(\mu,\nu)<\infty$. Then we
  have
  \begin{align}\label{eq:duality2}
    W^2_2(\mu,\nu)~=~\sup\left\{\int\phi^c \dd \nu + \int \phi\dd \mu~:~\phi\in\mathsf{Lip}_b(\cs)\cap C(\cs) \right\}\;,
  \end{align}
\end{corollary}

Recall the local Lipschitz constant from \eqref{def:locLip}. The next
proposition states that the Hopf--Lax semigroup yields a solution
of the Hamilton--Jacobi equation.

\begin{proposition}\cite[Thm. 3.6]{AGS11a}
  \label{prop:HJsubsol}
  For $\gamma\in\mathcal D(f)$ and $t\in(0,t_*(\gamma))$ it holds that
  \begin{align*}
    \frac{\dd}{\dd t}Q_tf(\gamma) +
    \frac{|DQ_tf(\gamma)|^2(\gamma)}{2}~=~ 0
  \end{align*}
  with at most countably many exceptions in $(0,t_*(\gamma))$.
\end{proposition}

\section{Manifestations of curvature on $\cs$}
\label{sec:curv-manifest}

In this section we derive several curvature properties of the
configuration space by ``lifting'' the corresponding statement from
the base manifold $M$ to $\cs$.

\subsection{Sectional curvature bounds}
\label{sec:alex}

We start by showing that the configuration space inherits Alexandrov
curvature bounds from the base space.

By Toponogov's triangle comparison theorem a lower bound on the
sectional curvature of a Riemannian manifold can be characterized by a
condition involving only the distance function. This allows to
generalize the notion of sectional curvature bounds to metric spaces
and gives rise to Alexandrov spaces. Loosely put, an Alexandrov space
with curvature bounded below by $K\in\R$ is a complete length space
$(X,d)$ in which triangles are ``thicker'' than in the space form of
constant curvature $K$. We refer to \cite{BBI} for a nice and
comprehensive treatment of Alexandrov geometry. There are various
equivalent ways of characterizing Alexandrov curvature. We will use
the following taken from \cite{LP10}:

\begin{definition}\label{def:Alexcurv}
  A complete length space $(X,d)$ is an Alexandrov space with
  curvature bounded below by $K\in\R$ iff the following holds: For
  each quadruple of points $x_0,x_1,x_2,x_3\in X$ we have:
  \begin{align}\nonumber
    \sum_{i=1}^3 d^2(x_0,x_i)
     ~&\geq~
     \frac16 \sum_{i,j=1}^3 d^2(x_i,x_j)\;, & \text{if } K=0\;,\\\label{eq:Alexdef}
    \left(\sum_{i=1}^3\cosh\big(\sqrt{\abs{K}}d(x_0,x_i)\big)\right)^2
     ~&\geq~
    \sum_{i,j=1}^3 \cosh\big(\sqrt{\abs{K}}d(x_i,x_j)\big)\;, &\text{if } K<0\;.\\\nonumber 
    \left(\sum_{i=1}^3 \cos\big(\sqrt{K}d(x_0,x_i)\big)\right)^2
    ~&\leq~ 
    \sum_{i,j=1}^3 \cos\big(\sqrt{K}d(x_i,x_j)\big)\;,& \text{if } K>0\;.
  \end{align}
\end{definition}

\begin{remark}
  There is a variant of this characterization by Sturm,
  \cite{Sturm99}. The proof of Theorem \ref{thm:sec-bounds} adapts
  with only minor changes.
\end{remark}

Note in particular that the Riemannian manifold $M$ has sectional
curvature bounded below by $K$ if and only if its Riemannian distance $d$
satisfies \eqref{eq:Alexdef}. Definition \ref{def:Alexcurv} does not
apply immediately to extended metric spaces such as the configuration
space $(\cs,\dcs)$. However, considering the fibers $\cs_\sigma
:=\{\gamma\in\cs:\dcs(\gamma,\sigma)<\infty\}$, we note that
$(\cs_\sigma,\dcs)$ is a complete length metric space for each
$\sigma\in\cs$.

\begin{theorem}\label{thm:sec-bounds}
  Assume that the base manifold $M$ has sectional curvature bounded
  below by $K\in\R$. Then (any fiber of) $(\cs,d_\cs)$ is an
  Alexandrov space with curvature bounded below by $\min\{K,0\}$ in
  the sense of Definition \ref{def:Alexcurv}.
\end{theorem}


\begin{proof}
  We will only consider the case $K<0$, the case $K=0$ follows by
  similar arguments or alternatively can be obtained from this by
  letting $K\nearrow0$. Obviously the case $K>0$ is reduced
  immediately to $K=0$. We will verify the quadruple comparison
  inequality. So let $\gamma_0,\gamma_1,\gamma_2,\gamma_3\in\cs$ such
  that $d_\cs(\gamma_0,\gamma_i)<\infty$ for $i=1,2,3$ (and hence also
  $d_\cs(\gamma_i,\gamma_j)<\infty$). In particular, we have
  $\gamma_i(M)=C$ for all $i=0,1,2,3$ and some
  $C\in\N\cup\{+\infty\}$. We will assume $C=+\infty$, the case
  $C<\infty$ follows from the same arguments and is simpler. Using
  Lemma \ref{lem:existence of matching} we can choose labelings
  $\gamma_i=\sum_n \delta_{x^i_n}$ for $i=0,1,2,3$ such that
  \begin{align}\label{eq:d1}
    d^2_\cs(\gamma_0,\gamma_i) ~=~ \sum\limits_{n=1}^\infty d^2\big(x^0_n,x^i_n\big) ~<~ \infty\;.
  \end{align}
  Further we can estimate for $i,j=1,2,3$:
  \begin{align}\label{eq:d2}
    d^2_\cs(\gamma_i,\gamma_j) ~\leq~ \sum\limits_{n=1}^\infty d^2\big(x^i_{n},x^j_{n}\big) ~<\infty~\;,
  \end{align}
  where finiteness follows from the triangle inequality in $(M,d)$ and
  \eqref{eq:d1}. Using the fact that for any $N\in\N$ the product
  manifold $M^N$ with Riemannian distance
  $d^2_N\big((x_1,\cdots,x_N),(y_1,\cdots,y_N)\big) = \sum_{n=1}^N
  d^2(x_n,y_n)$ has sectional curvature bounded below by $K$ and thus
  satisfies quadruple comparison, we get setting $\lambda=\sqrt{\abs{K}}$:
  \begin{align*}
    \left(\sum_{i=1}^3 \cosh\big(\lambda d_\cs(\gamma_0,\gamma_i)\big)\right)^2
    ~&=~ \lim\limits_{N\to\infty} \left(\sum_{i=1}^3 \cosh\Big(\lambda \sqrt{\sum_{n=1}^Nd^2(x^0_{n},x^i_{n})}\Big)\right)\\
    ~&\geq~ \lim\limits_{N\to\infty} \sum_{i,j=1}^3 \cosh\Big(\lambda \sqrt{\sum_{n=1}^Nd^2(x^i_{n},x^j_{n})}\Big)\\
    ~&=~ \sum_{i,j=1}^3 \cosh\Big(\lambda \sqrt{\sum_{n=1}^\infty d^2(x^i_{n},x^j_{n})}\Big)\\
    ~&\geq~ \sum_{i,j=1}^3 \cosh\big(\lambda d_\cs(\gamma_i,\gamma_j)\big)\;,
  \end{align*}
  where the last inequality follows from \eqref{eq:d2} and the fact
  that $\cosh$ is increasing. This finishes the proof.
\end{proof}

\subsection{Bochner inequality on configuration space}
\label{sec:bochner}

Starting from this section we will be concerned with lower bounds on
the Ricci curvature. Let us recall the Bochner--Weitzenb\"ock identity
which asserts that for every smooth function $u:M\to\R$ on the
Riemannian manifold $M$ we have:
\begin{align*}
  \frac12\Delta\abs{\nabla u}^2 -\ip{\nabla u,\nabla\Delta u} = \norm{\Hess u}^2_{HS} + \Ric[\nabla u,\nabla u]\;,
\end{align*}
where $\norm{\cdot}_{HS}$ denotes the Hilbert--Schmidt norm and $\Ric$
denotes the Ricci tensor. Thus a lower bound on the Ricci curvature in
the form $\Ric[\nabla u,\nabla u]\geq K\abs{\nabla u}^2$ is seen to be
equivalent to the \emph{Bochner inequality}
\begin{align*}
  \frac12\Delta\abs{\nabla u}^2 -\ip{\nabla u,\nabla\Delta u} \geq K \abs{\nabla u}^2\;.
\end{align*}

It will be convenient to introduce the \emph{carr\'e du champ
  operators}, defined for smooth functions $\phi,\psi:M\to\R$ via
\begin{align*}
  \Gamma(\phi,\psi)   ~&:=~ \frac12 \left[ \Delta \big(\phi\psi\big) - \phi \Delta\psi -\psi \Delta\phi \right]  ~=~ \ip{\nabla\phi,\nabla\psi}\;,\\
  \Gamma_2(\phi,\psi) ~&:=~ \frac12 \left[ \Delta \Gamma(\phi,\psi) - \Gamma(\phi,\Delta\psi) \Gamma(\psi,\Delta\phi) \right]\;.
\end{align*}
In particular, writing $\Gamma(\phi)=\Gamma(\phi,\phi)$ and
$\Gamma_2(\phi)=\Gamma_2(\phi,\phi)$ we see $\Gamma_2(\phi)= \frac12
\Delta\abs{\nabla \phi}^2 -\ip{\nabla\phi,\nabla\Delta\phi}$.  Thus
the Bochner inequality takes the form
\begin{align*}
  \Gamma_2(\phi) ~\geq~ K\ \Gamma(\phi)\;.
\end{align*}
The latter inequality has been used extensively in the study of
general Markov semigroups and diffusions, originating in the work of
{\sc Bakry--\'Emery} \cite{BE85}, where $\Delta$ is replaced by the
generator of the semigroup.

The aim of this section is to prove the natural analogue of Bochner's
inequality on the configuration space. For smooth cylinder functions
$F,G\in\cyl$ we define
\begin{align*}
  \Gamma^\cs (F,G) ~&:=~ \frac12\left[\Delta^\cs(FG)- F\Delta^\cs G -G\Delta^\cs F\right]~=~ \ip{\nabla^\cs F,\nabla^\cs G}\;,\\
  \Gamma^\cs_2(F,G) ~&:=~ \frac12 \left[\Delta^\cs \Gamma^\cs(F,G) - \Gamma^\cs(F,\Delta^\cs G) - \Gamma^\cs(G,\Delta^\cs F)\right]\;.
\end{align*}
Note that $\Gamma^\cs$ coincides with the carr\'e du champ operator of
the Dirichlet form $\cE$ introduced in Section \ref{sec:diff}.

\begin{proposition}\label{prop:Bochner-cylinder}
  Assume that $M$ has Ricci curvature bounded below
  by $K$. Then any cylinder function $F\in\cyl$
  satisfies the following Bochner inequality:
  \begin{align}\label{eq:Bochner-cylinder}
    \Gamma_2^\cs (F)(\gamma) ~\geq K\ \Gamma^\cs (F)(\gamma)
    \qquad\forall \gamma\in\cs\;.    
  \end{align}
\end{proposition}

\begin{proof}
  The cylinder function $F$ takes the form
  $F(\gamma)=g\big(\ip{\phi_1,\gamma},\dots,\ip{\phi_n,\gamma}\big)$,
  where $g\in C^\infty(\R^n)$ and $\phi_i\in C^\infty_c(M)$ for
  $i=1,\dots, n$. From the definition of gradient and divergence on
  $\cs$ a direct calculation yields:
  \begin{align*}
    \Gamma^\cs (F)(\gamma) ~&=~\sum\limits_{i,j}g_i(\phi) g_j(\phi)\ip{\nabla\phi_i,\nabla\phi_j}_\gamma
                         ~=~ \sum\limits_{i,j}g_i(\phi) g_j(\phi) \ip{\Gamma(\phi_i,\phi_j),\gamma}\;,
  \end{align*}
  where we write $g_i=\partial_ig$. Moreover, we obtain
  \begin{align*}
    &\Gamma_2^\cs(F)(\gamma)\\
    ~&=~\sum\limits_{i,j} g_i(\phi)g_j(\phi)\ip{\frac12\Delta\ip{\nabla \phi_i,\nabla\phi_j} - \ip{\nabla \phi_i,\nabla\Delta \phi_j},\gamma}\\
    &\qquad + \sum\limits_{i,j,k,l}g_{ik}(\phi)g_{jl}(\phi)\ip{\nabla\phi_i,\nabla\phi_j}_\gamma\ip{\nabla\phi_k,\nabla\phi_l}_\gamma\\
    &\qquad +
    \sum\limits_{i,j,k}g_{i}(\phi)g_{jk}(\phi)\Big[2\ip{\nabla\ip{\nabla\phi_i,\nabla\phi_k},\nabla\phi_j}_\gamma
    -
    \ip{\nabla\ip{\nabla\phi_j,\nabla\phi_k},\nabla\phi_i}_\gamma\Big]\\
     &=~ \sum\limits_{i,j} g_i(\phi)g_j(\phi)\ip{\Gamma_2(\phi_i,\phi_j),\gamma} 
        + \sum\limits_{i,j,k,l}g_{ik}(\phi)g_{jl}(\phi)\ip{\Gamma(\phi_i,\phi_j),\gamma}\ip{\Gamma(\phi_k,\phi_l),\gamma}\\
         &\qquad+ \sum\limits_{i,j,k}g_{i}(\phi)g_{jk}(\phi)\Big[\ip{2\Gamma\big(\phi_j,\Gamma(\phi_i,\phi_k)\big),\gamma}
                                               -\ip{\Gamma\big(\phi_i,\Gamma(\phi_j,\phi_k)\big),\gamma}\Big]\;.
  \end{align*}
  Choose a compact set $K$ containing all the supports of $\phi_i$ for
  $i=1,\dots,n$. Fix a configuration $\gamma$, let $N=\gamma(K)$ and
  write $\gamma|_K=\sum_{\a=1}^N\delta_{x_\a}$. Define functions
  $\psi_i:M^N\to\R$ via
  $\psi_i(y_1,\cdots,y_N)=\sum_{\a=1}^N\phi_i(y_\a)=\ip{\phi_i,\gamma}$. By
  the tensorization property \eqref{eq:Gamma-tensor} and the chain
  rule \eqref{eq:Gamma2-chain} of the carr\'e du champ operators given
  by Lemma \ref{lem:Gamma2-tensor-chain} below we obtain for
  $\xx=(x_1,\cdots,x_N)$:
  \begin{align*}
    \Gamma_2^\cs(F)(\gamma)
    ~&=~ \sum\limits_{i,j} g_i(\psi)g_j(\psi)\Gamma^{(N)}_2(\psi_i,\psi_j)(\xx)\\ 
        &+ \sum\limits_{i,j,k,l}g_{ik}(\psi)g_{jl}(\psi)\Gamma^{(N)}(\psi_i,\psi_j)(\xx)\Gamma^{(N)}(\psi_k,\psi_l)(\xx)\\
         &+ \sum\limits_{i,j,k}g_{i}(\psi)g_{jk}(\psi)\Big[2\Gamma^{(N)}\big(\psi_j,\Gamma^{(N)}(\psi_i,\psi_k)\big)(\xx)\\
                                               &\qquad\qquad-\Gamma^{(N)}\big(\psi_i,\Gamma^{(N)}(\psi_j,\psi_k)\big)(\xx)\Big]\\
      &=~ \Gamma_2^{(N)}(g(\psi))(\xx)\;.
  \end{align*}
  Applying Bochner's inequality on $M^N$, which has Ricci curvature
  bounded below by $K$ as well, and using \eqref{eq:Gamma-tensor},
  \eqref{eq:Gamma-chain} we get:
  \begin{align*}
    \Gamma_2^\cs(F)(\gamma) ~&=~ \Gamma_2^{(N)}(g(\psi))(\xx)
     ~\geq~ K\ \Gamma^{(N)}(g(\psi))(\xx)\\ 
    ~&=~ \sum\limits_{i,j} g_i(\phi)g_j(\phi)\ip{\Gamma(\phi_i,\phi_j),\gamma}
    ~=~ 
    K\ \Gamma^{\cs}(F)(\gamma)\;,
  \end{align*}
  which finishes the proof.
\end{proof}

The following lemma summarizes tensorization properties and a chain
rule for the carr\'e du champ operators which are readily verified by
direct computations.

\begin{lemma}\label{lem:Gamma2-tensor-chain}
  Let $M$ be a smooth Riemannian manifold. Let $g\in C^\infty(\R^n)$
  and $\psi_i\in C^\infty_c(M)$ for $i=1,\dots,n$ and write
  $\psi=\big(\psi_1,\cdots,\psi_n\big)\in C^\infty_c(M,\R^n)$. Then we
  have:
  \begin{align}\label{eq:Gamma-chain}
    \Gamma\big(g(\psi)\big) ~&=~ \sum\limits_{i,j=1}^n g_i(\psi) g_j(\psi) \Gamma(\psi_i,\psi_j)\;,\\\nonumber
    \Gamma_2\big(g(\psi)\big) ~&=~ \sum\limits_{i,j=1}^n g_i(\psi) g_j(\psi) \Gamma_2(\psi_i,\psi_j)
                                + \sum\limits_{i,j,k,l=1}^n g_{ik}(\psi)g_{jl}(\psi) \Gamma(\psi_i,\psi_j)\Gamma(\psi_k,\psi_l)\\\label{eq:Gamma2-chain}
                              & + \sum\limits_{i,j,k=1}^ng_i(\psi) g_{jk}(\psi)
                              \left[2\Gamma\big(\psi_j,\Gamma(\psi_i,\psi_k)\big) - \Gamma\big(\psi_i,\Gamma(\psi_j,\psi_k)\big)\right]\;.
  \end{align}
  Moreover, for $N\in\N$ let $M^N$ be the $N$-fold tensor product of
  the Riemannian manifold $M$ and denote by
  $\Gamma^{(N)},\Gamma_2^{(N)}$ the carr\'e du champ operators
  associated to the Laplace--Beltrami operator on $M^N$. Let
  $\psi:M^N\to\R$ be given for $\xx=(x_1,\cdots,x_N)$ by
  $\psi(\xx)=\sum_{\a=1}^N\phi(x_\a)$ for a function $\phi\in
  C^\infty_c(M)$. Then we have:
  \begin{align}\label{eq:Gamma-tensor}
    \Gamma^{(N)}(\psi) (\xx) ~=~ \sum\limits_{\a=1}^N \Gamma(\phi)(x_\a)\;,\qquad
     \Gamma_2^{(N)}(\psi)(\xx) ~=~ \sum\limits_{\a=1}^N \Gamma_2(\phi)(x_\a)\;.
  \end{align}
\end{lemma}

More generally we have the following weak form of Bochner's inequality.

\begin{proposition}\label{prop:Bochner}
  Assume that $\Ric_M\geq K$. Then for all non-negative $G\in
  D(\Delta^\cs)$ with $G,\abs{\nabla^\cs G},\Delta^\cs G\in L^\infty(\cs,\pi)$ and all
  $F\in D(\Delta^\cs)$ we have:
  \begin{align}\label{eq:Bochner-weak}
    \int \frac12\Delta^\cs G \abs{\nabla^\cs F}^2  + G(\Delta^\cs F)^2 + \Delta^\cs F\ip{\nabla^\cs G,\nabla^\cs F}\dd \pi
    ~\geq~ 
   K \int G\abs{\nabla^\cs F}^2\dd\pi\;.
  \end{align}
\end{proposition}

\begin{proof}
  First let $F$ be a cylinder function. Multiplying
  \eqref{eq:Bochner-cylinder} by $G$ and integrating we obtain \eqref{eq:Bochner-weak} immediately by applying the Leibniz rule
  \begin{align*}
    G\ip{\nabla^\cs F,\Delta^\cs F} = \ip{\nabla^\cs F, \nabla^\cs(G\Delta^\cs
      F)}  - \Delta^\cs F\ip{\nabla^\cs G,\nabla^\cs F}
  \end{align*}
  and an integration by parts. For general $F\in
  D(\Delta^\cs)\subset\cF$ we argue by approximation. We can take a
  sequence $(F_n)\subset\cyl$ such that $F_n\to F$, $\abs{\nabla^\cs
    F_n}\to\abs{\nabla^\cs F}$ and $\Delta^\cs F_n\to\Delta^\cs F$ in
  $L^2(\cs,\pi)$. By the boundedness of $G,\abs{\nabla^\cs G}$ and
  $\Delta^\cs G$ we can pass to the limit in the integrals and obtain
  \eqref{eq:Bochner-weak}.
\end{proof}

\subsection{Gradient estimates on $\cs$}
\label{sec:gradest}

It is well known that the lower curvature bound $\Ric_M\geq K$ is
equivalent to the following gradient estimate for the heat semigroup
$T^M_t=\e^{t\Delta}$ on $M$, see
e.g. \cite[Thm. 1.3]{vRenesseSturm:gradientestimates} and the
discussion thereafter. For all smooth $f:M\to\R$, all $x\in M$ and $t>0$:
\begin{align}\label{eq:M-gradest}
  \Gamma\big( T_t^M f\big)(x) ~\leq~ \e^{-2Kt}T_t^M\Gamma\big(f\big)(x)\;.
\end{align}
The aim of this section is to show the gradient estimate for the heat
semigroup $T_t^\cs$ on the configuration space. Recall that the
Dirichlet form admits a carr\'e du champs operator $\Gamma^\cs$ such
that for all $u\in\cF$ we have $\Gamma^\cs(u)(\gamma)=\abs{ \nabla^\cs
  u}^2_\gamma$. We have the following

\begin{theorem}\label{thm:config-gradest}
  Assume that $\Ric_M\geq K$. Then for any function $F\in \cF$ and all
  $t>0$ we have:
  \begin{align}\label{eq:config-gradest}
    \Gamma^\cs \big(T^\cs_t F\big) ~\leq~ \e^{-2Kt} T^\cs_t\Gamma^\cs\big( F\big) \quad \pi\text{-a.e.}
  \end{align}
\end{theorem}

The strategy we follow will be to use the explicit representation of
the semigroup $T_t^\cs$ as an infinite product of one-particle
semigroups and the tensorization property of the gradient
estimate. Before we give the proof we need to introduce some notation.

Recall that $\pi(\cs^{(\infty)})=1$. To a measurable function $F$ on
$\cs^{(\infty)}$ we associate $\hat F: M^\N\to\R$ via
\begin{align*}
  \hat F(\xx) := F\left(\sum_{i\geq 1}\delta_{x_i}\right),
  \quad \xx=(x_i)_{i\geq 1}\in M^\N\;,
\end{align*}
which is measurable with respect to the product $\sigma-$algebra on
$M^\N.$ Then, also the function $\hat F_{\xx}^i : M\to\R$ defined by
\begin{align*}
  \hat F^i_\xx (y):= F\left(\sum_{j\geq 1, j\neq i} \delta_{x_j} +
  \delta_y\right)
\end{align*}
is measurable. We say that $\hat F$ is differentiable in $\xx$
if for each $i\geq 1$ the gradient in the i-th direction
\begin{align*}
  \nabla^i \hat F(\xx):= \nabla \hat F^i_\xx(x_i)\;,
\end{align*}
exists. We say that $\hat F$ is differentiable with finite gradient if
additionally
\begin{align*}
  |\nabla^\N \hat F|^2(\xx):=\sum_{i\geq 1} |\nabla^i\hat
  F|_{x_i}^2(\xx)<\infty\;.
\end{align*}
Then, for every $F\in\cyl, \gamma\in\cs^{(\infty)}$ and $\xx\in
l^{-1}(\gamma)$ we have
\begin{align*}
  \Gamma^\cs(F)=|\nabla^\cs F|_\gamma^2=|\nabla^\N\hat F|^2(\xx)\;.
\end{align*}
We will put
\begin{align*}
  T^i_t\hat F(\xx) = T_t^M \hat F^i_\xx(x_i)\;,
\end{align*}
i.e.\ the action of the one-particle semigroup in the $i$-th
coordinate. With this notation we can express the semigroup $T^\N_t$
introduced in Section \ref{sec:heat-sg} as $T^\N_t=\Pi_{j\in\N}T^j_t$,
the iterated application of the one-particle semigroup in all
directions. For $i\in\N$ we will also put
\begin{align*}
  T^{\check i}_t = \prod\limits_{j\in \N, j\neq i}T^j_t\;.
\end{align*}

\begin{proof}[Proof of Theorem \ref{thm:config-gradest}]
  Let us first assume that $F\in\cyl$ and start by establishing a
  gradient estimate for $\hat F$. First note that by
  \eqref{eq:M-gradest} for any $i\in\N$ and $\xx\in M^\N$:
  \begin{align*}
    \abs{\nabla^iT^i_t\hat F}^2(\xx) ~=~ \abs{\nabla T^M_t\hat F^i_\xx}^2(x_i)
      ~\leq~ \e^{-2Kt}T^M_t\abs{\nabla\hat F_\xx^i}^2(x_i)
      ~=~   \e^{-2Kt}T^i_t\abs{\nabla^i\hat F}^2(\xx)\;.
  \end{align*}
  By Jensen's inequality this yields
  \begin{align*}
    \abs{\nabla^iT^\N_t\hat F}^2(\xx) 
       ~\leq~ T_t^{\check i}\abs{\nabla^iT^i_t\hat F}^2(\xx) ~\leq~ \e^{-2Kt}T_t^\N\abs{\nabla^i\hat F}^2(\xx)\;,
  \end{align*}
  and summing over $i$ we obtain
  \begin{align}\label{eq:ge1}
    \abs{\nabla^\N T^\N_t\hat F}^2(\xx)
    ~\leq~
    \e^{-2Kt}T_t^\N\abs{\nabla^\N \hat F}^2(\xx)
    ~ <~ \infty\;.
  \end{align}
  In particular $T_t^\N \hat F$ is differentiable with finite
  gradient. Note that the right hand side is also bounded above by a
  constant. We now want to pass from the estimate on $M^\N$ to an
  estimate on $\cs$. Note that for any good configuration
  $\gamma\in\Theta$ and $\xx\in l^{-1}(\gamma)$:
  \begin{align*}
    T_t^\N\abs{\nabla^\N \hat F}^2(\xx) ~=~ \tilde T_t^\cs\abs{\nabla^\cs F}^2(\gamma) ~=:~G(\gamma)\;.
  \end{align*}
  We claim that $\tilde T^\cs_tF$ is $d_\cs$-Lipschitz on $\Theta$ and
  that $\abs{D\tilde T_t^\cs F}\leq \e^{-2Kt}G$. By Lemma
  \ref{lem:lip-gamma} this will suffice to show
  \eqref{eq:config-gradest}. Indeed, consider $V\in\mathcal V_0(M)$
  and its flow $(\psi_t)_t$. Then we have
  \begin{align*}
    \abs{\tilde T^\cs_tF(\psi_1^*\gamma)-\tilde T^\cs_tF(\gamma)} 
     ~&\leq~ \left|\int_0^1 \frac{\dd}{\dd s}\tilde T^\cs_tF(\psi_s^*\gamma)\dd s\right|\\
     ~&=~ \left|\int_0^1\sum_i\ip{\nabla^iT^\N_t\hat F,V}(\psi_s^*\xx)\dd s\right|\\
     ~&\leq~ \e^{-Kt}\int_0^1 \sqrt{G(\psi_s^*\gamma)}\abs{V}_{\psi_s^*\gamma}\dd s\\
     ~&=~ d_\cs(\psi_1^*\gamma,\gamma)\e^{-Kt}\int_0^1 \sqrt{G(\psi_s^*\gamma)}\dd s\;.
   \end{align*}
   Thus $\tilde T^\cs_t F$ is Lipschitz by the boundedness of
   $G$. Arguing as in the proof of Lemma \ref{lem:lip-gamma} by
   letting $\abs{V}_\gamma\to0$ yields the claim by continuity of $G$.

   Now take $F\in\cF$. Then there is a sequence $(F_n)_{n\in\N}\subset
   \cyl$ such that $F_n\to F$ in $L^2(\cs,\pi)$ and
   $\cE(F-F_n)\to0$. Therefore, denoting by $\Lambda$ the measure
   $\Lambda(dx,d\gamma):=\gamma(dx)\pi(d\gamma)$, $\nabla^\cs F_n$ is
   a Cauchy sequence in $L^2(M\times\cs\to TM,\Lambda).$ Therefore,
   there is a limit, denoted by $\nabla^\cs F,$ such that
   $\cE(F)=\int\abs{\nabla^\cs F}^2\ \dd\pi.$ As $T_t^\cs$ is a
   contraction also $T_t^\cs F_n \to T_t^\cs F$,
   $T_t^\cs\abs{\nabla^\cs F_n}^2\to T_t^\cs\abs{\nabla^\cs F_n}^2$
   and $\nabla^\cs T^\cs_t F_n$ is a Cauchy sequence with some limit
   $G$, by \eqref{eq:config-gradest}. By lower semicontinuity of the
   carr\'e du champ operator (see e.g. \cite[(2.17)]{AGS12}) we have
   $\Gamma^\cs(T_t^\cs F)(\gamma)\leq |G|^2(\gamma)\ \pi$-a.e.. In the
   first part of the proof, we saw that \eqref{eq:config-gradest}
   holds for all $F_n$, i.e.\
  \begin{align*}
    \abs{\nabla^\cs T^\cs_t F_n}^2 \leq \e^{-2Kt}
    T^\cs_t(\abs{\nabla^\cs F_n}^2)\quad \pi\mbox{-a.e.}
  \end{align*}
  Extracting a subsequence, this yields
  \begin{align*}
    \abs{\nabla^\cs T^\cs_t F}^2(\gamma) \leq |G|^2(\gamma)\leq
    \e^{-2Kt} T^\cs_t(\abs{\nabla^\cs F}^2)(\gamma)\quad
    \pi\mbox{-a.e.}
  \end{align*}
\end{proof}

\begin{remark}\label{rem:Bochner-Gradest}
  Alternatively, the gradient estimate could have been derived from
  the Bochner inequality from the previous section. In fact, a
  classical interpolation argument due to {\sc Bakry--\'Emery} yields the
  equivalence of the $\Gamma_2$-inequality \eqref{eq:Bochner-cylinder}
  and the gradient estimate \eqref{eq:config-gradest}. The idea is to
  consider
  \begin{align*}
    \phi(s)=\e^{-2Ks}T_s^\cs\Gamma^\cs(T_{t-s}^\cs F)
  \end{align*}
  and note that
  $\phi'(s)=\e^{-2Ks}T_s^\cs\Big[\Gamma_2^\cs(T_{t-s}^\cs F) - K
  \Gamma^\cs(T_{t-s}^\cs)\Big]$. For a detailed proof in a general
  setting see e.g. \cite[Cor. 2.3]{AGS12}. However, in order to apply
  this in the present setting one would need to extend
  \eqref{eq:Bochner-cylinder} (in a weak form) to a larger class of
  functions.
\end{remark}

\subsection{Wasserstein contraction}
\label{sec:W2contract}

In \cite{vRenesseSturm:gradientestimates} it has been shown that a
lower bound on the Ricci curvature is also equivalent to expansion
bounds in Wasserstein distance for the heat kernel. More precisely,
\cite[Cor. 1.4]{vRenesseSturm:gradientestimates} states that
$\Ric_M\geq K$ if and only if 
\begin{align}\label{eq:WcontM}
  W_{p,d}(H^M_t\mu,H^M_t\nu) ~\leq~ \e^{-Kt}W_{p,d}(\mu,\nu) \quad \forall t>0,\  \mu,\nu\in\cP_p(M)\;.
\end{align}
Here $W_{p,d}$ denotes the $L^p$-Wasserstein distance built from the
Riemannian distance $d$ and $H^M_t\mu \in\cP(M)$ is the probability
measure defined by
\begin{align*}
 (H^M_t\mu)(A) = \int_A\int_M p^M_t(x,y)\dd\mu(x)\dd\vol(y)\qquad\forall A\in\cB(\cs)\;,
\end{align*}
where $p^M_t$ is the heat kernel on $M$.

Here we will show that the heat semigroup on the configuration space
has the corresponding expansion bound in Wasserstein distance provided
$\Ric_M\geq K$. Recall the set of good configurations $\Theta$ from
\eqref{good config}. A probability measure $\mu$ on $\cs$ is called
\emph{good} if it is concentrated on the good configurations, i.e.\
$\mu(\Theta)=1.$ We denote the set of all good probability measures by
$\cP_g(\cs).$ Note that in particular any measure absolutely
continuous w.r.t.\ $\pi$ and all Dirac measures
$\delta_\gamma\in\cP(\cs)$ with $\gamma\in\Theta$ are good.  Let $p^\cs_t$
be the semigroup of Markov kernels on $\Theta$ given by Theorem
\ref{thm:sg-welldef} (i.e.\ the transition probabilities of the
independent particle process). Given $\mu\in\cP_g(\cs)$ we define $H^\cs_t\mu$ via
\begin{align}\label{eq:dualsg-good}
  H^\cs_t\mu(A) = \int_\Theta p^\cs_t(\gamma,A)\dd\mu(\gamma)\;. 
\end{align}

\begin{theorem}
  \label{thm:W2contract}
  Assume that $\Ric_M\geq K$. Then for all $\mu,\nu\in \cP_g(\cs)$ we
  have:
  \begin{align}\label{eq:W2cont}
    W_{2,\dcs}(H^\cs_t\mu,H^\cs_t\nu)~\leq~ \e^{-Kt}W_{2,\dcs}(\mu,\nu)\quad\forall t>0\;.
  \end{align}
\end{theorem}

\begin{proof}
  First we show that for all $\gamma,\sigma\in \Theta$ we have:
  \begin{align}\label{eq:dcscont}
    W_{2,\dcs}\big(p^\cs_t(\gamma,\cdot),p^\cs_t(\sigma,\cdot)\big) ~\leq~ \e^{-Kt}\dcs(\gamma,\sigma)\quad\forall t>0\;.
  \end{align}
  We can assume that $\dcs(\gamma,\sigma)<\infty$ and consider only
  the case $\gamma(M)=\sigma(M)=\infty$. Then by Lemma
  \ref{lem:existence of matching} there exist labelings
  $\gamma=\sum_{i=1}^\infty\delta_{x_i}$ and
  $\sigma=\sum_{i=1}^\infty\delta_{y_i}$ such that
  $\dcs^2(\gamma,\sigma)=\sum_{i=1}^\infty d^2(x_i,y_i)$. Now, for any
  $i$ choose an optimal coupling $q_i\in \cP(M\times M)$ of
  $p^M_t(x_i,\cdot)$ and $p^M_t(y_i,\cdot)$ such that
  \begin{align*}
    W_{2,d}^2 \big(p^M_t(x_i,\cdot),p^M_t(y_i,\cdot)\big) = \int d^2(u,v)\dd q_i(u,v)\;. 
  \end{align*}
  Let $q^\N=\bigotimes_{i=1}^\infty q_i\in \cP(M^\N\times M^\N)$ and
  set $q=(l\times l)_\#q^\N$, where $l$ is the labelling map. Then
  $q\in\cP(\Theta\times\Theta)$ defines a coupling of
  $p^\cs_t(\gamma,\cdot)$ and $p^\cs_t(\sigma,\cdot)$. Now we can estimate:
  \begin{align*}
    W^2_{2,\dcs}\big(p^\cs_t(\gamma,\cdot),p^\cs_t(\sigma,\cdot)\big) ~&\leq~ \int \dcs^2\dd q ~=~ \int \dcs^2\big(l(\boldsymbol{u}),l(\boldsymbol{v})\dd q^\N(\boldsymbol{u},\boldsymbol{v})\\
   ~&\leq \sum_{i=1}^\infty \int  d^2(u_i,v_i)\dd q_i(u_i,v_i) ~=~ \sum_{i=1}^\infty W^2_{2,d}\big(p^M_t(x_i,\cdot),p^M_t(y_i,\cdot)\big)\\
   ~&\leq~ \e^{-2Kt}\sum_{i=1}^\infty d^2(x_i,y_i) ~=~ \e^{-2Kt}\dcs^2(\gamma,\sigma)\;.
  \end{align*}
  Here we have estimated $\dcs$ by the choice of a special labeling in
  the second inequality and used \eqref{eq:WcontM} in the third
  inequality.  Finally, to prove \eqref{eq:W2cont} we can again assume
  that $W_{2,\dcs}(\mu,\nu)<\infty$ and choose an optimal coupling
  $q$ of $\mu$ and $\nu$. Then by convexity of the squared
  Wasserstein distance we get
  \begin{align*}
    W_{2,\dcs}^2(H^\cs_t\mu,H^\cs_t\nu) ~&\leq~ \int W^2_{2,\dcs}\big(p^\cs_t(\gamma,\cdot),p^\cs_t(\sigma,\cdot)\big)\dd q(\gamma,\sigma)\\
                                 ~&\leq~ \e^{-2Kt} \int \dcs^2(\gamma,\sigma)\dd q(\gamma,\sigma) ~=~ \e^{-2Kt}W^2_{2,\dcs}(\mu,\nu)\;. 
  \end{align*}
 \end{proof}

\begin{remark}
  The same argument as in the previous proof yields that for any
  $p\in[1,\infty]$ and any $\mu,\nu\in\cP_g(\cs)$:
  \begin{align*}
   W_{p,d_{\cs,p}}(H^\cs_t\mu,H^\cs_t\nu)~\leq~ \e^{-Kt}W_{p,d_{\cs,p}}(\mu,\nu)\quad\forall t>0\;,
  \end{align*}
  where $d_{\cs,p}$ is the $L^p$-transport distance between non-normalized measures.

  Moreover, combining the construction of the semigroup in
  \eqref{eq:naivsgcs} with
  \cite[Cor. 1(x)]{vRenesseSturm:gradientestimates} one can show along
  the lines of the previous proof that for any two good configurations
  $\gamma$ and $\sigma$ there exist a coupling
  $(\bB^\gamma_t,\bB^\sigma_t)$ of the two copies of the independent
  particle process in $\Theta$ on some probability space
  $(\Omega,\cF,\P)$ starting in $\gamma$ respectively $\sigma$ such
  that
  \begin{align*}
    \dcs(\bB^\gamma_t,\bB^\sigma_t)\leq \e^{-Kt}\ \dcs(\gamma,\sigma) \quad
    \P\mbox{ a.s.}
  \end{align*}
\end{remark}

\section{Synthetic Riemannian Ricci curvature}
\label{sec:RCD}

It has been proven in \cite{vRenesseSturm:gradientestimates} that $M$
has Ricci curvature bounded below by $K$, if and only if the entropy
is $K$-convex along geodesics in $\big(\cP_2(M),W_2\big)$. This result
has been the starting point for {\sc Sturm} \cite{sturm:gmms1} and
{\sc Lott--Villani} \cite{LV09} to define a notion of Ricci curvature for
metric measure spaces.

The goal of this section is to show that the configuration space
satisfies (a version for extended metric measure spaces of) this so-called CD$(K,\infty)$
condition, provided $\Ric_M\geq K$. Unlike the previous results we
will not obtain this by ``lifting'' the corresponding statement from
the base space. Instead we will follow the approach in \cite{AGS12}
and derive the so-called Evolution Variational Inequality starting from
the gradient estimates established in Theorem
\ref{thm:config-gradest}. This will yield geodesic convexity as an
immediate consequence and as a side product give the characterization
of the heat semigroup on $\cs$ as the gradient flow of the entropy.

The argument will follow closely the lines of \cite[Sec. 4]{AGS12}. A
careful inspection of the proofs given there in the case of a
Dirichlet form with finite intrinsic distance, reveals that most of
them carry over to the present setting of an extended metric measure
space. We give a sketch of the arguments in Section \ref{sec:action}
to make this transparent. However, in the configuration space setting
we need to work significantly more to establish the required
regularization properties of the heat semigroup. This is the purpose
of Section \ref{sec:dualsg}. We assume from now on that $\Ric_M\geq
K$.

\subsection{Regularizing properties of the dual semigroup}
\label{sec:dualsg}

We denote the set of probability measures absolutely continuous 
w.r.t.\ $\pi$ by $\cP_{ac}(\cs)$. Given $\mu\in\cP_{ac}(\cs)$ with
$\mu=f\pi$ we define the action of the dual semigroup $H_t^\cs$ via
\begin{align*}
  H_t^\cs\mu ~=~ (T_t^\cs f)\pi\;.
\end{align*}
Note that this coincides with $H_t^\cs\mu$ defined for good probability
measures $\mu$ in \eqref{eq:dualsg-good}. Thanks to the Wasserstein
contractivity \eqref{eq:W2cont} we can extend $H_t^\cs$ to a contractive
semigroup on the closure of $\cP_{ac}(\cs)$ w.r.t.\ $W_2$. 

Given $\mu\in\cP_{ac}(\cs)$ with $\mu=f\pi$ the \emph{relative
  entropy} w.r.t.\ $\pi$ is defined by
\begin{align*}
  \ent(\mu) ~=~ \int f\log f\dd\pi\;.
\end{align*}
If $\mu$ is not absolutely continuous we set $\ent(\mu)=\infty$. Note
that $\ent(\mu)\geq0$ for all $\mu\in\cP(\cs)$ since $\pi$ is a
probability measure. We write $D(\ent)=\{\mu:\ent(\mu)<\infty\}$. We
will denote by $\cP_e$ the set of all probability measures whose fiber
contains a measure of finite entropy,
\begin{align*}
  \cP_e = \{\mu\in\cP(\cs) : \exists \nu\in D(\ent),\ W_2(\mu,\nu)<\infty\}\;.
\end{align*}

\begin{lemma}
  \label{lem:ac-approx}
  For any $\mu\in\cP_e$ there exists a sequence of probability
  measures $(\mu_n)_n\in D(\ent)$ with $W_2(\mu_n,\mu)\to0$ as
  $n\to\infty$. In particular, $H_t^\cs\mu$ is defined for any such
  $\mu$.
\end{lemma}

The proof relies on an explicit construction but is rather lengthy and
we postpone it to the appendix. Note that $\cP_e$ is obviously the
maximal set of measures that can be approximated in this way. To prove
regularization properties of $H_t^\cs$ we need to collect some
estimates. The \emph{Fisher information} of $\mu=f\pi$ with
$\sqrt{f}\in\cF$ is defined via
\begin{align*}
  I(\mu):= 4 \cE(\sqrt{f})\;,
\end{align*}
Otherwise we set $I(\mu)=\infty$. Note that we can also write
\begin{align*}
  I(\mu)~=~\int_{\{f>0\}}\frac{\Gamma^\cs(f)}{f}\dd\pi\;,
\end{align*}
and that $I$ is convex on $\cP_{ac}(\cs)$, see
\cite[Prop. 4.1]{AGS12}. 

A curve $\mu:J\to\cP(\cs)$ is called \emph{$p$-absolutely continuous}
w.r.t.\ $W_2$ on an interval $J$, written $\mu\in
AC^p\big(J,(\cP(\cs),W_2)\big)$, if there exist $a\in L^p(J,\Leb)$ such
that for all $s,t\in J$:
\begin{align*}
  W_2(\mu_s,\mu_t)~\leq~\int_s^ta(r)\Leb(\dd r)\;.
\end{align*}
For any absolutely continuous curve $\mu:J\to\cP(\cs)$ the \emph{metric derivative}
defined by
\begin{align*}
  \abs{\dot\mu_t} ~=~ \lim\limits_{h\to0} \frac{W_2(\mu_{t+h},\mu_t)}{h}
\end{align*}
exits for a.e. $t\in J$, see \cite[Thm. 1.1.2]{AGS08}.

Having identified the Dirichlet form $\cE$
with the Cheeger energy $\Ch$ constructed from $\dcs$ in Proposition
\ref{prop:E=2Ch} yields in particular that $T_t^\cs$ coincides with
the gradient flow of $\Ch$ in $L^2(\cs,\pi)$. This allows us to apply
useful estimates for this gradient flow established in \cite{AGS11a}.

\begin{lemma}\label{lem:dissipation}
  Let $\mu=f\pi$ with $\ent(\mu)<\infty$ and set $\mu_t=(T^\cs_t
  f)\pi$. Then the map $t\mapsto\ent(\mu_t)$ is non-increasing,
  locally absolutely continuous. Moreover, we have for all $T>0$:
 \begin{align}\label{eq:Fisherbound1}
   \int_0^T I(\mu_t)\dd t ~\leq~ 2\ent(\mu_0)\;.
 \end{align}
 The curve $t\mapsto \mu_t$ is absolutely continuous w.r.t.\ $W_2$ and
 for a.e. $t$:
 \begin{align}\label{eq:speedFisher}
   \abs{\dot\mu_t}^2 ~\leq~ I(\mu_t)\;.
 \end{align}
\end{lemma}

\begin{proof}
  That the entropy is non-increasing and \eqref{eq:Fisherbound1} holds for
  $f\in L^1(\pi)\cap L^2(\pi)$ are proven in \cite[Lem. 4.19,
  Prop. 4.22]{AGS11a}. The general statement follows by a truncation
  argument using the lower semicontinuity of $I$ in $L^1(\cs,\pi)$ and
  of $\ent$ w.r.t.\ weak convergence (and thus also in
  $L^1(\cs,\pi)$). Finally \cite[Lem.\ 6.1]{AGS11a} gives
  \eqref{eq:speedFisher}.
\end{proof}

The following $\log$-Harnack and entropy--cost inequalities will be
crucial for the regularizing properties of the dual semigroup.

\begin{lemma}\label{l:entropycost}
  For any bounded Borel-measurable function $f:\cs\to\R$ all $t>0$ and
  $\gamma,\sigma\in\Theta$ we have:
  \begin{align}\label{eq:Harnack1}
    (\tilde T^\cs_t\log f)(\gamma)~\leq~ \log\big(\tilde T^\cs_{t}f(\sigma)\big) + \frac{K}{2(1-\e^{-2Kt})}\ d^2_\cs(\gamma,\sigma)\;. 
  \end{align}
  In particular, for any $\mu\in\cP_{ac}(\cs)$ and $\nu\in D(\ent)$
  we have:
 \begin{align}\label{eq:Harnack2}
   \ent(H_t^\cs\mu) ~\leq~ \ent(\nu) + \frac{K}{2(1-\e^{-2Kt})}\ W_2^2(\mu,\nu)\;.
\end{align}
\end{lemma}

\begin{proof}
  \eqref{eq:Harnack1} is proven in \cite[ Theorem 2.2]{Deng2014} for
  $f\geq1$. For general measurable $f\geq 0$ we apply this result to
  $f_\eps=\eps^{-1}(f+\eps)\geq1$ and obtain
  \begin{align*}
    \big(\tilde T^\cs_t\log (f+\epsilon)\big)(\gamma)~\leq~ \log\big(\tilde T^\cs_{t}(f(\sigma)+\epsilon)\big) + \frac{K}{2(1-\e^{-2Kt})}\ d^2_\cs(\gamma,\sigma)\;. 
  \end{align*}
  Letting $\eps\to0$ yields \eqref{eq:Harnack1}. To prove
  \eqref{eq:Harnack2} consider $\mu=f\pi\in\cP_{ac}(\cs)$ and
  $\nu=g\pi\in D(\ent)$ with $W_2(\mu,\nu)<\infty$. Applying
  \eqref{eq:Harnack1} with $\tilde T^\cs_t f$ and integrating against
  an optimal coupling $q$ of $H_t^\cs\mu=(\tilde T_t^\cs f)\pi$ and
  $\nu$ we obtain:
  \begin{align*}
    \ent(H_t^\cs\mu) ~&=~ \int \tilde T_t^\cs f \log \tilde T_t^\cs f \dd\pi\\
                 &\leq~ \int \big(\log \tilde T_{2t}^\cs f\big)\dd \nu  + \frac{K}{2(1-\e^{-2Kt})}\ W_2^2(\mu,\nu)\;. 
  \end{align*}
  Using Jensen's inequality and the fact that $\int \tilde T_{2t}^\cs
  f\dd\pi=1$ we estimate:
  \begin{align*}
    \int \big(\log \tilde T_{2t}^\cs f\big)\dd \nu ~&=~ \int \log\left( \frac{\tilde T_{2t}^\cs f}{g}\right)\dd\nu + \int \log g\dd \nu\\
                                                  &\leq~ \log\left(\int  \frac{\tilde T_{2t}^\cs f}{g}\dd \nu\right) + \ent(\nu) ~=~ \ent(\nu)\;,
  \end{align*}
  which proves the claim.
\end{proof}

Consider the following mollification of the semigroup, defined for
$\eps>0$ and $f \in L^p (\cs,\pi)$, $1 \le p \le \infty$ via:
\begin{align}\label{eq:sg-moll}
  h^\eps
  f~=~\int_0^\infty\frac{1}{\eps}\eta\left(\frac{t}{\eps}\right)\tilde T^\cs_tf\,\dd
  t\;,
\end{align}
with a non-negative kernel $\eta\in C^\infty_c(0,\infty)$ satisfying
$\int_0^\infty\eta(t)\dd t=1$. Combining the convexity of $I$ with
\eqref{eq:Harnack2} and \eqref{eq:Fisherbound1} we obtain that for all
$t>0$ and all non-negative $f\in L^1(\cs,\pi)$ the measure
$\mu=(h^\eps f)\pi$ satisfies (see also \cite[Lem. 4.9]{AGS12}):
\begin{align}\label{eq:moll-bounds}
  I(H_t^\cs\mu) ~\leq~ C(\eps)\Big(W^2_2(f\pi,\nu)+\ent(\nu)\Big)\;,
\end{align}
where the constant $C(\eps)$ on the right hand side depends only on
$\eps$.

\begin{lemma}\label{lem:Pt-approx}
  For any $\mu\in\cP_e$ and $t>0$ we have $H_t^\cs\mu\in D(\ent)$,
  $W_2(H_t^\cs\mu,\mu)<\infty$ and moreover, $W_2(H_t^\cs\mu,\mu)\to 0$ as
  $t\to 0$.
\end{lemma}

\begin{proof}
  First assume that $\mu=f\pi$ and $\ent(\mu)<\infty$ and set
  $\mu_t=H_t^\cs\mu$. Since $H_t^\cs$ decreases the entropy we have
  also $\mu_t\in D(\ent)$. Further, we obtain by
  \eqref{eq:speedFisher} and H\"older's inequality:
  \begin{align*}
    W_2^2(H_t^\cs\mu,\mu) ~\leq~ \int_0^t\abs{\dot\mu_s}\dd s ~\leq~ \sqrt{t} \left(\int_0^t I(\mu_s)\dd s\right)^{\frac12}\;,
  \end{align*}
  which goes to zero as $t\to0$ by \eqref{eq:Fisherbound1} and thus
  the lemma is established for $\mu\in D(\ent)$.

  Now consider the general case where $\mu$ does not belong to
  $D(\ent)$. By Lemma \ref{lem:ac-approx}, we can approximate it in
  $W_2$ by measures $\mu_n\in D(\ent)$. By Lemma \ref{l:entropycost}
  we have for some $\nu\in D(\ent)$ and all $n$
  \begin{align*}
    \ent(H_t^\cs\mu_n)\leq \ent(\nu) + \frac{K}{2(1-\e^{-2Kt})} W_2(\mu_n,\nu)\;.
  \end{align*}
  The left hand side is uniformly bounded in $n$ because
  $W_2(\mu_n,\mu)\to 0.$ Hence, the entropies stay bounded as
  well. Moreover, by the Wasserstein contractivity of $H_t^\cs$ we have
  \begin{align*}
    W_2(H_t^\cs \mu_n,H_t^\cs\mu) \leq \e^{-2Kt}\ W_2(\mu_n,\mu)\to 0\;,
  \end{align*}
  implying the weak convergence of $H_t^\cs\mu_n$ to $H_t^\cs\mu$. By
  lower semicontinuity of the entropy we can derive in the limit
  $n\to\infty$:
  \begin{align*}
    \ent(H_t^\cs\mu)\leq  \liminf \ent(H_t^\cs\mu_n)\leq & \ent(\nu) + \frac{K}{2(1-\e^{-2Kt})} W_2(\mu,\nu)<\infty\;.
  \end{align*}
  Finally, from the triangle inequality together with Wasserstein
  contraction we infer:
  \begin{align*}
    W_2(H_t^\cs\mu,\mu) &\leq W_2(H_t^\cs\mu,H_t^\cs\mu_n) + W_2(H_t^\cs\mu_n,\mu_n) +W_2(\mu_n,\mu)\\
    &\leq (1+\e^{-2Kt}) W_2(\mu,\mu_n) + W_2(H_t^\cs\mu_n,\mu_n)\;,
  \end{align*}
  The right hand side can be made arbitrarily small by first choosing
  $n$ so big such that the first term is small uniformly for
  $t\in[0,1]$ and then taking $t$ small to make the second term small
  by the first part of the proof. This proves the last claim of the
  lemma.
\end{proof}

We will now describe the regularization procedure needed in the
sequel. We will use the notion of \emph{regular curve} as introduced
in \cite[Def.~4.10]{AGS12}. Briefly, we call a curve $(\mu_s)_{s\in[0,1]}$
with $\mu_s=f_s \pi$ regular if the following are satisfied:
\begin{itemize}
\item $(\mu_s)$ is $2$-absolutely continuous in $(\cP(\cs),W_2)$,
\item $\ent(\mu_s)$ and $I(H_t^\cs\mu_s)$ are bounded for $s\in[0,1], t\in[0,T]$ for any $T>0$,
\item $f\in C^1\big([0,1],L^1(\cs,\pi)\big)$ and $\Delta^{(1)}f\in C\big([0,1],L^1(\cs,\pi)\big)$,
\item $f_s=h^\eps\tilde f_s$ for some $\tilde f_s\in L^1(\cs,\pi)$ and $\eps>0$.
\end{itemize}
Here $\Delta^{(1)}$ denotes the generator of the semigroup $T^\cs_t$
in $L^1(\cs,\pi)$ and $h^\eps$ is the mollification of the semigroup
 introduced in \eqref{eq:sg-moll}.

 In the sequel we will denote by $\dot f_s$ the derivative of
 $[0,1]\ni s\mapsto f_s\in L^1(\cs,\pi)$. We will need the following
 result which is an adaption and slight improvement of
 \cite[Prop.~4.11]{AGS12}.

\begin{lemma}[Approximation by regular curves]\label{lem:reg-curves}
  Let $(\mu_s)_{s\in[0,1]}$ be a $2$-absolutely continuous curve in
  $\big(\cP(\cs),W_2\big)$ such that $\mu_s\in\cP_e$ for some (hence
  any) $s\in[0,1]$. Then there exists a sequence of regular curves
  $(\mu_s^n)$ with the following properties. As $n\to\infty$ we have
  for any $s\in[0,1]$:
\begin{align}\label{eq:reg-curves1}
  W_2(\mu^n_s,\mu_s)~&\to~0\;,\\\label{eq:reg-curves1a}
   \limsup \abs{\dot\mu^n_s}~&\leq~\abs{\dot\mu_s}\quad\text{a.e. in }[0,1]\;.
 \end{align}
 Moreover, if $\ent(\mu_0),\ent(\mu_1)<\infty$ we have:
 \begin{align}\label{eq:reg-curves2}
  \ent(\mu^n_0)\to\ent(\mu_0)\;,\quad  \ent(\mu^n_1)\to\ent(\mu_1)\;.   
 \end{align}
\end{lemma}

\begin{proof}
  Following \cite[Prop.~4.11]{AGS12} we employ a threefold
  regularization procedure. Given $n$, we construct a curve
  $(\mu^{n,0}_s)_s$ with $s\in[-\frac1n,1+\frac1n]$ by setting
  \begin{align*}
    \mu^{n,0}_s~=~
    \begin{cases}
      \mu_0\;, & -\frac1n\leq s \leq \frac1n\;,\\
      \mu_{(s-\frac1n)/(1-\frac2n)}\;, & \frac1n\leq s\leq 1-\frac1n\;,\\       
      \mu_1\;, & 1-\frac1n \leq s \leq 1+\frac1n\;.
    \end{cases}
  \end{align*}
   
  Then, for $s\in[0,1]$ we first define
  $\mu^{n,1}_s=H_{1/n}\mu^{n,0}_s=f^{n,1}_s\pi$, which is absolutely
  continuous w.r.t.\ $\pi$ by Lemma \ref{lem:Pt-approx}. The second
  step consists in a convolution in the time parameter. We set
    \begin{align*}
    \mu^{n,2}_s=f_s^{n,2}\pi\;,\qquad f^{n,2}_s~=~\int f^{n,1}_{s-s'}\psi_{n}(s')\dd s'\;,
  \end{align*}
  where $\psi_n(s)=n\cdot\psi(n s)$ for some smooth kernel
  $\psi:\R\to\R_+$ supported in $[-1,1]$ with $\int\psi(s)\dd
  s=1$. Finally, we set
  \begin{align*}
    \mu^{n}_s~=~f^n_s \pi\;,\qquad f^n_s~=~h^{1/n} f^{n,2}_s\;,
  \end{align*}
  where $h^\eps$ denotes a mollification of the semigroup given by
  \eqref{eq:sg-moll}. Following the argument in
  \cite[Prop.~4.11]{AGS12} one sees that $(\mu^n_s)_{s\in[0,1]}$
  constructed in this way is a regular curve and that
  \eqref{eq:reg-curves1} holds. Note that in our setting the
  convergence \eqref{eq:reg-curves1} relies on Lemma
  \ref{lem:ac-approx}, the uniform bounds on entropy and Fisher
  information are ensured by the $L\log L$-regularization
  \eqref{eq:Harnack2} and the estimate
  \eqref{eq:moll-bounds}. \eqref{eq:reg-curves1a} follows from the
  convexity properties of $W_2^2$ and the $K$-contractivity of the
  heat flow. To prove \eqref{eq:reg-curves2}, simply note that for
  $i=0,1$
  \begin{align*}
    \ent(\mu^n_i)=\ent(h^{1/n}H_{1/n}\mu_i) \leq \ent(\mu_i)
  \end{align*}
  since $H_t^\cs$ and hence also $h_t$ decreases the entropy by Lemma
  \ref{lem:dissipation}. This together with \eqref{eq:reg-curves1} and
  lower semicontinuity of $\ent$ implies \eqref{eq:reg-curves2}.
 \end{proof}

\subsection{Action estimate}
\label{sec:action}

Here we establish the key action estimate, Proposition
\ref{prop:action-est}, which allows us to derive the Evolution
Variational Inequality in the next section. 

We proceed very closely along the lines of \cite[Sec. 4.3]{AGS12}
where the corresponding result has been proven in the setting of a
Dirichlet form with a finite intrinsic metric. However, a careful
inspection of the proofs reveals that the same arguments work almost
verbatim in the present context of an extended intrinsic distance. We
give a sketch of the main steps in the argument to make the line of
reasoning transparent. We refer to \cite[Sec.\ 4.3]{AGS12} for detailed
proofs.

For the following lemmas let $(\mu_s)_{s\in[0,1]}$ be a regular curve
and write $\mu_s=f_s\pi$. We set
$\mu_{s,t}=H_{st}^\cs\mu_s=f_{s,t}\pi$. Moreover, let $\phi:\cs\to\R$
be bounded and $\dcs$-Lipschitz. We set $\phi_s=Q_s\phi$ for
$s\in[0,1]$, where
\begin{align*}
  Q_s\phi(\gamma)~:=~\inf\limits_{\sigma\in \cs} \left[ \phi(\sigma) + \frac{\dcs^2(\gamma,\sigma)}{2s} \right]
\end{align*}
denotes the Hopf-Lax semigroup as recalled in Section \ref{sec:dual and HL}.

Following \cite[Lem.\ 4.13]{AGS12} we first obtain the following
estimate.

\begin{lemma}
  \label{lem:derivative-HL}
  For any $t>0$ the map $s\mapsto \int\phi_s\dd\mu_{s,t}$ is
  absolutely continuous and we have for a.e $s\in(0,1)$:
  \begin{align}\label{eq:derivatve-HL}
    \frac{\dd}{\dd s}\int\phi_s\dd\mu_{s,t} ~=~ \int\dot
    f_sT^\cs_{st}\phi_s\dd\pi -\frac12\int\Gamma^\cs(\phi_s)\dd\mu_{s,t} - t
    \int 2\sqrt{f_{s,t}}\Gamma^\cs\big(\sqrt{f_{s,t}},\phi_s\big)\dd\pi\;.
  \end{align}
\end{lemma}

We need to use a regularization $E_\eps$ of the entropy functional where the
singularities of the logarithm are truncated. Let us define
$e_\eps:[0,\infty)\to\R$ by setting
$e_\eps'(r)=\log(\eps+r\wedge\eps^{-1} )+1$ and $e_\eps(0)=0$. Then
for any $\mu=f\pi\in\cP(\cs)$ we define
\begin{align*}
  E_\eps(\mu):=\int e_\eps(f)\dd \pi\;.
\end{align*}
Moreover we set $p_\eps(r)=e_\eps'(r^2)-\log\eps - 1$. Note that for any
$\mu\in D(\ent)$ we have $E_\eps(\mu)\to\ent(\mu)$ as $\eps\to0$.

Following \cite[Lem.\ 4.15]{AGS12}, we obtain an estimate for the
derivative of the regularized entropy $E_\eps$ along the curve
$s\mapsto\mu_{s,t}$.
\begin{lemma}
  \label{lem:derivative-ent}
 For any $t>0$ we have
  \begin{align}\label{eq:derivative-ent}
    E_\eps(\mu_{1,t}) - E_\eps(\mu_{0,t})
   ~\leq~
   \int_0^1\left[\int T^\cs_{st}(g^\eps_{s,t})\dot f_s\dd\pi - t \int\Gamma^\cs(g^\eps_{s,t})\dd \mu_{s,t} \right]\dd s\;,
  \end{align}
   where we put $g^\eps_{s,r}=p_\eps(\sqrt{f_{s,r}})$.
\end{lemma}

The following estimate follows parallel to \cite[Lem.\ 4.12]{AGS12}
building on Lisini's theorem for extended metric spaces from
\cite{Li14}.

\begin{lemma}
  \label{lem:speed}
  For any curve $(\mu_s)_{s\in[0,1]}$ in
  $AC^2\big([0,1],(\cP(\cs),W_2)\big)$ with $\mu_s=f_s \pi$ and $f\in
  C^1\big((0,1),L^1(\cs,\pi)\big)$ and any $\dcs$-Lipschitz function
  $\phi$ we have
\begin{align}\label{eq:speed-est}
  \left|\int \dot f_s \phi\dd \pi\right|~\leq~\abs{\dot\mu_s}\cdot\sqrt{\int \Gamma^\cs(\phi)f_s\dd \pi}\;.
\end{align}  
\end{lemma}

Now we can establish the action estimate by following \cite[Thm.\ 4.16]{AGS12}.

\begin{proposition}
  \label{prop:action-est}
  For any regular curve $(\mu_s)_{s\in[0,1]}$ in $\cP(\cs)$ and all
  $t>0$ we have:
  \begin{align}\label{eq:action-est}
   \frac12 W^2_2(\mu_{1,t},\mu_{0,t}) - \frac12 \int_0^1\e^{-2Kst}\abs{\dot\mu_s}^2\dd s
    ~\leq~ t \left[\ent(\mu_{0,t})-\ent(\mu_{1,t}) \right]\;. 
  \end{align}
\end{proposition}

\begin{proof}
  Fix a function $\phi:\cs\to\R$ which is bounded and $\dcs$-Lipschitz.
  Applying Lemma \ref{lem:derivative-HL} and Lemma
  \ref{lem:derivative-ent} we first obtain
  \begin{align*}
    &\int \phi_1\dd\mu_{1,t} - \int \phi_0\dd\mu_{0,t} 
    + t \left[E_\eps(\mu_{1,t})-E_\eps(\mu_{0,t}) \right] 
    - \int_0^1\frac12\e^{-2Kst}\abs{\dot\mu_{s}}^2\dd s\\
     &=~
     \int_0^1 \frac{\dd}{\dd s}\left[ \int \phi_s\dd \mu_{s,t} + t E_\eps(\mu_{s,t})  \right]  - \frac12\e^{-2Kst}\abs{\dot\mu_{s}}^2 \dd s\\
     &\leq~
    \int_0^1 \left[\int \dot f_s T^\cs_{st}(\phi_s + t g_{s,t}^\eps)\dd \pi -\frac12\e^{-2Kst}\abs{\dot\mu_{s}}^2\right.\\
    &\qquad \left. - t^2\int \Gamma^\cs(g^\eps_{s,t})\dd\mu_{s,t}
     -\frac12 \int \Gamma^\cs(\phi_s)\dd\mu_{s,t} - t \int 2\sqrt{f_{s,t}}\Gamma^\cs\big(\sqrt{f_{s,t}},\phi_s\big)\dd\pi\right]\dd s\\
    &=:~A+B\;,
  \end{align*}
  where $A$ and $B$ denote the sums of the terms in the first and
  second line respectively. Let us put
  $q_\eps(r)=\sqrt{r}\big(2-\sqrt{r}p_\eps'(\sqrt{r})\big)$. Then we
  have by the chain rule 
  \begin{align*}
   2\sqrt{f_{s,t}}\Gamma^\cs(\sqrt{f_{s,t}},\phi_s)~=~f_{s,t}\Gamma^\cs(g^\eps_{s,t},\phi_s)+q_\eps(f_{s,t})\Gamma^\cs(\sqrt{f_{s,t}},\phi_s)\;. 
  \end{align*}
  Using this and completing the square we obtain
  \begin{align}\label{eq:action-B}
    B~\leq~   
    \int_0^1 \left[- \frac12 \int \Gamma^\cs(\phi_s + tg^\eps_{s,t})\dd\mu_{s,t}
    -t \int q_\eps(f_{s,t})\Gamma^\cs\big(\sqrt{f_{s,t}},\phi_s\big)\dd\pi\right] \dd s\;.
  \end{align}
  Using \eqref{eq:speed-est}, Young's inequality as well as the
  gradient estimate \eqref{eq:config-gradest} from Theorem
  \ref{thm:config-gradest} we infer that
   \begin{align}\nonumber
    A~&\leq~
    \int_0^1 \left[  \frac12 \e^{2Kst}\int \Gamma^\cs\big(T^\cs_{st}(\phi_s + t g^\eps_{s,t})\big) f_s \dd \pi \right]\dd s\\
    \label{eq:action-A}
     &\leq~ \int_0^1\left[ \frac12 \int \Gamma^\cs\big(\phi_s+tg^\eps_{s,t}\big)\dd\mu_{s,t} \right]\dd s\;.
   \end{align}
   Combining \eqref{eq:action-A} and \eqref{eq:action-B} we obtain
   that for any $\delta>0$: 
 \begin{align*}
   A + B 
   ~&\leq~ \int_0^1\left[-t \int q_\eps(f_{s,t})\Gamma^\cs\big(\sqrt{f_{s,t}},\phi_s\big)\dd\pi\right] \dd s\\
    &\leq~ t \int_0^1 \int \abs{q_\eps(f_{s,t})}\sqrt{\Gamma^\cs\big(\sqrt{f_{s,t}}\big)\Gamma^\cs\big(\phi_s\big)}\dd\pi \dd s\\
    &\leq~ \int_0^1 \left[ \frac{t\delta}{8} I(\mu_{s,t}) + \frac{t}{2\delta}\int q^2_\eps(f_{s,t})\Gamma^\cs\big(\phi_s\big)\dd\pi \right] \dd s\;,
 \end{align*}
 where we have used Young's inequality again. Now, using that
 $q_\eps^2(r)\leq r$ and $q_\eps(r)\to0$ as $\eps\to0$ we can pass to the
 limit first as $\eps\to 0$ and then as $\delta\to 0$ to arrive at
 \begin{align*}
   \int\phi_1\dd \mu_{1,t} - \int\phi_0\dd\mu_{0,t} - \frac12 \int_0^1\e^{-2Kst}\abs{\dot\mu_s}^2\dd s
    ~\leq~ t \left[\ent(\mu_{0,t})-\ent(\mu_{1,t}) \right]\;. 
 \end{align*}
 Finally, taking the supremum with respect to $\phi$ and invoking the
 Kantorovich duality Cor. \ref{cor:duality} we get
 \eqref{eq:action-est}.
\end{proof}

\subsection{EVI, geodesic convexity and gradient flows}
\label{sec:evi}

We can now prove the main result of this section.

\begin{theorem}
  \label{thm:RCD}
  Assume that $\Ric_M\geq K$. Then the dual heat semigroup
  $(H_t^\cs)_t$ satisfies the following Evolution Variational
  Inequality. For all $\sigma\in D(\ent)$ and $\mu\in\cP(\cs)$ with
  $W_2(\mu,\sigma)<\infty$:
  \begin{align}\label{eq:evi}
    \ddtr \frac12 W_2^2(H_t^\cs\mu,\sigma) + \frac{K}2 W_2^2(H_t^\cs\mu,\sigma)~\leq~\ent(\sigma)-\ent(H_t^\cs\mu)\quad\forall t>0\;.
  \end{align}
\end{theorem}

Here we denote by
\begin{align*}
  \frac{\dd^+}{\dd t}f(t)~=~\limsup_{h\searrow 0}\frac{f(t+h)-f(t)}{h}
\end{align*}
the upper right derivative.

\begin{proof}
  By Lemma \ref{lem:ac-approx} we have that $H_t^\cs\mu$ is well
  defined, belongs to $D(\ent)$ and $W_2(H_t^\cs\mu,\sigma)<\infty$
  for all $t\geq0$. By the semigroup property it is sufficient to
  assume $\mu\in D(\ent)$ and prove \eqref{eq:evi} at $t=0$. Let
  $(\mu_s)_s$ be a curve in
  $AC^2\left([0,1],\big(\cP(\cs),W_2\big)\right)$ connecting
  $\mu_0=\sigma$ to $\mu_1=\mu$. By Lemma \ref{lem:reg-curves} we can
  find approximating regular curves $(\mu^n_s)_s$ and applying
  Proposition \ref{prop:action-est} to the curves
  $\mu^n_{s,t}=H_{st}^\cs\mu_{s}^n$ we find:
  \begin{align*}
    \frac12 W^2_2(\mu^n_{1,t},\mu^n_{0,t}) - \frac12 \int_0^1\e^{-2Kst}\abs{\dot\mu^n_s}^2\dd s
    ~\leq~ t \left[\ent(\mu^n_{0,t})-\ent(\mu^n_{1,t}) \right]\;.
  \end{align*}
  Passing to the limit $n\to\infty$ and using the convergences
  \eqref{eq:reg-curves1}, \eqref{eq:reg-curves1a} and
  \eqref{eq:reg-curves2} as well as lower semicontinuity of $\ent$ we
  get:
  \begin{align*}
    \frac12 W^2_2(H_t^\cs\mu,\sigma) - \frac12 \int_0^1\e^{-2Kst}\abs{\dot\mu_s}^2\dd s
   ~\leq~ t \left[\ent(\sigma)-\ent(H_t^\cs\mu) \right]\;.
  \end{align*}
  Minimizing over the curve $(\mu_s)_s$ and using the fact that
  $(\cP(\cs),W_2)$ is a length space we obtain
  \begin{align*}
    \frac12 W^2_2(H_t^\cs\mu,\sigma) - \frac12\frac{\e^{2Kt}-1}{2Kt}W_2^2(\mu,\sigma)
    ~\leq~ t \left[\ent(\sigma)-\ent(H_t^\cs\mu) \right]\;.
  \end{align*}
  Dividing by $t$ and letting $t\searrow0$ finally yields \eqref{eq:evi}.
\end{proof}

As a direct consequence we obtain convexity of the entropy along
geodesics.

\begin{corollary}\label{cor:ent-convex}
  For all $\mu_0,\mu_1\in D(\ent)$ with $W_2(\mu_0,\mu_1)<\infty$ and
  any geodesic $(\mu_s)_{s\in[0,1]}$ connecting them we have for all
  $s\in[0,1]$:
  \begin{align}
    \label{eq:ent-convex}
    \ent(\mu_s)~\leq~(1-s)\ent(\mu_0) + s \ent(\mu_1) - \frac{K}{2}s(1-s)W^2_2(\mu_0,\mu_1)\;.
  \end{align}
\end{corollary}

\begin{proof}
  This follows from the very same argument as in
  \cite[Thm. 3.2]{DS08}. Since all the distances appearing are finite,
  the fact that we deal with extended metric spaces does not play a
  role. To make this clear we give a sketch of the
  proof. 
  
  Multiplying \eqref{eq:evi} with $\e^{Kt}$ and integrating from $0$
  to $t$ yields that for every $\sigma\in D(\ent)$ and
  $\mu\in\cP(\cs)$ with $W_2(\mu,\sigma)<\infty$:
  \begin{align*}
    \frac{\e^{Kt}}{2}W^2_2(H_t^\cs\mu,\sigma) - \frac12 W^2_2(\mu,\sigma)
    ~\leq~ \frac{\e^{Kt}-1}{K}\Big(\ent(\sigma)-\ent(H_t^\cs\mu)\Big)\;.
  \end{align*}
  Applying this with $\mu=\mu_s$ and $\sigma=\mu_0$ or $\sigma=\mu_1$
  respectively and taking a convex combination of the resulting
  inequalities we get
  \begin{align*}
    &\frac{\e^{Kt}-1}{K}\Big((1-s)\ent(\mu_0) + s\ent(\mu_1) - \ent(H_t^\cs\mu_s)\Big)\\
    &\geq~ \frac{\e^{Kt}}{2}\Big((1-s)W_2^2(H_t^\cs\mu_s,\mu_0) + s W_2^2(H_t^\cs\mu_s,\mu_1) \Big)\\
    & - \frac12\Big((1-s)W_2^2(\mu_s,\mu_0) + s W_2^2(\mu_s,\mu_1) \Big)\\
    &\geq~\frac{\e^{Kt}-1}{2} s(1-s)W_2^2(\mu_0,\mu_1)\;.
  \end{align*}
  In the last step we have used the elementary inequality
  \begin{align*}
    (1-s)a^2 + sb^2~\geq~s(1-s)(a+b)^2 \quad \forall a,b>0,\ s\in[0,1]\;,
  \end{align*}
  the triangle inequality and the fact that $(\mu_s)_s$ is a constant
  speed geodesic. Dividing by $\e^{Kt}-1$ and letting $t\searrow
  0$ then yields \eqref{eq:ent-convex}.
\end{proof}

\begin{remark}
  We have obtained that $(\cs,\dcs,\pi)$ is an extended metric measure
  space satisfying the CD$(K,\infty)$ curvature bound in the sense of
 {\sc Lott--Villani} and {\sc Sturm}, see also \cite[Def. 9.1]{AGS11a} for an
  extension of the definition to extended metric measure
  spaces. Moreover, it is a \emph{strong} CD$(K,\infty)$ space in the
  sense that convexity holds along all geodesics.
\end{remark}

\section{Appendix}\label{sec:app}

{\bf Proof of Lemma \ref{lem:ac-approx}}
Given $\mu\in \cP_e$ we will construct a sequence of measures
$\mu_n\in D(\ent)$ such that $W_2(\mu_n,\mu)\to0$ as $n\to\infty$. So
let us fix $\nu\in D(\ent)$ with $W_2(\mu,\nu)<\infty$. The strategy
of the proof is to find a big bounded set in the base space $M$ in
which most of the transport happens. In this set we can approximate
the measure $\mu$ nicely. Outside this set we will keep the
$\nu$-points to end up with a measure in the support of the entropy.
\medskip\\
{\bf Construction of $\mu_n$:}

Fix $x_0\in M$ and $n\in\N$ and set $B=B(x_0,n)$. Choose an optimal
coupling $q\in\Opt(\mu,\nu)$. By Lemma \ref{lem:existence of matching}
we can choose for each $(\gamma,\omega)\in\supp(q)$ an optimal
matching $\eta\in\Opt(\gamma,\omega)$. Bby Lemma \ref{lem:meas selec},
the map $(\gamma,\omega)\mapsto \eta$ can be chosen measurable.
Denoting by $\pr_i$ the projection onto the i-th component, define a
map $(\gamma,\omega)\mapsto\xi\in \cs$ via
\begin{align}\label{eq:xi}
  \xi ~:=~ \pr_1(\1_{B\times B}\eta) \cup \pr_2(\1_{\complement B\times
    \complement B}\eta) \cup\pr_2(\1_{\complement B\times B}\eta)
  \cup\pr_2(\1_{B\times \complement B}\eta)
\end{align}
For a Borel set $V\subset M$ we define the restriction map
$r_V:\cs\to\cs$ by $r_V(\gamma)=\gamma|_V$. We will often use the
short hand notation $r_V(\gamma)=\gamma_V.$ By construction we have
\begin{align}\label{eq:Poisson on complement}
  \xi_{\complement B} ~&=~ \omega_{\complement B}\;,\\\label{eq:same mass in B}
  \xi(B) ~&=~\omega(B) \;.
\end{align}
Let us set $\alpha :=1/(2\sqrt{n\xi(B)})$. For $x\in \xi\cap B$ put
\begin{align*}
  \chi(x) = \begin{cases}
    x\;, & \text{ if } d(x,\complement B)>\alpha\;, \\
    x_\alpha\;, & \text{ otherwise }\;,
  \end{cases}
\end{align*}
where $x_\alpha$ is the point on the geodesic between $x$ and $x_0$
satisfying $d(x,x_\alpha)=\alpha$. Clearly $B(\chi(x),\alpha)\subset
B$. Denote the uniform distribution on $B(x,\alpha)$ by
$U_{x,\alpha}$.

Given $(\gamma,\omega)\in\supp(q)$ we define a probability measure
$\mathcal U^n_{\gamma,\omega}\in\cP(\cs)$ as follows. Let
$\xi(\gamma,\omega)$ be defined as in \eqref{eq:xi} and write
$\xi_B=\sum_{i=1}^k\delta_{x_i}$ and $\xi_{\complement
  B}=\sum_{i=k+1}^\infty\delta_{x_i}$. Given $(y_1,\dots,y_k)\in M$ we
put
\begin{align*}
  \tilde\xi(y_1,\dots,y_k) ~=~\sum_{i=1}^k\delta_{y_i}+\sum_{i=k+1}^\infty\delta_{x_i}\in\cs\;.
\end{align*}
Then we define
\begin{align*}
  \mathcal U^n_{\gamma,\omega}~:=~
  \int \Pi_{i=1}^k\delta_{\tilde\xi(y_1,\dots,y_k)}U_{\chi(x_i),\alpha}(\dd y_i)\;.
\end{align*}
Note that the map $T:(\gamma,\omega)\mapsto \mathcal
U^n_{\gamma,\omega}$ is measurable. We finally define
\begin{align*}
  \mu_n:= \int T(\gamma,\omega)\ q(d\gamma,d\omega) \ \in \cP(\cs)\;.
\end{align*}
The proof of Lemma \ref{lem:ac-approx} will be finished once we have
established the following claims.

\begin{claim}\label{cl:converge}
  $W_2(\mu,\mu_n)\to0$ as $n\to\infty$.
\end{claim}

\begin{claim}\label{cl:ent}
  For all $n$ we have $\Ent(\mu_n)<\infty$.
\end{claim}
   
\begin{proof}[Proof of Claim \ref{cl:converge}]
  Define for $\gamma, \omega\in \cs$
  \begin{align*}
    c_n(\gamma,\omega):= \inf\left\{ \int_{B_n\times B_n} d^2(x,y)\
      \eta(dx,dy), \ \eta\in\Opt(\gamma,\omega)\right\}\;.
  \end{align*}
  By the same reasoning as for Lemma \ref{lem:existence of matching}
  there is a matching realizing the infimum. By the compactness of
  $\Opt(\gamma,\omega)$ we have the pointwise convergence
  $c_n(\gamma,\omega)\nearrow c(\gamma,\omega)=\dcs^2(\gamma,\omega).$
  For $q\in \Opt(\mu,\nu)$ we have
  \begin{align*}
    \int c_n\ dq \nearrow \int c\ dq\;.
  \end{align*}
  For $\epsilon>0$ choose $n$ large enough such that
  \begin{align*}
    \int c\ dq - \epsilon \leq \int c_n\ dq\;.
  \end{align*}
  By construction we have for any $(\gamma,\omega)\in\supp(q)$ and
  $\xi=\xi(\gamma,\omega)$ as defined in \eqref{eq:xi}
  \begin{align}\label{eq:alpha disturb}
    W_2^2(\delta_{\xi}, \mathcal U^n_{\gamma,\omega})
    ~&\leq~
    4\xi(B)\alpha^2~=~ \frac1n\;,\\
   \label{eq:gamma tilde}
    W_2^2(\delta_\gamma,\delta_{\xi})
    ~&=~
    \dcs^2(\gamma,\xi) ~\leq~
    \dcs^2(\gamma,\omega) - c_n(\gamma,\omega)\;.
  \end{align}
  Consider the coupling $Q:=(\pr_1,T)_*q$ between $\mu$ and
  $\mu_n$. Using \eqref{eq:alpha disturb} and \eqref{eq:gamma tilde}
  and the convexity of $W_2^2$ we can deduce
  \begin{align*}
    W_2^2(\mu,\mu_n)
    ~&\leq~
    \int W_2^2(\delta_\gamma,\mathcal U^n_{\gamma,\omega})\dd q(\gamma,\omega)\\
    ~&\leq~
    \frac2n + 2\int c(\gamma,\omega)-c_n(\gamma,\omega) \dd
    q(\gamma,\omega)
    ~\leq~
    \frac2n + 2\epsilon\;,
    \end{align*}
   which finishes the proof.
  \end{proof}
  
  \begin{proof}[Proof of Claim \ref{cl:ent}]
    Note that \eqref{eq:Poisson on complement} implies that
    $(r_{\complement B})_*\mu_n= (r_{\complement
      B})_*\nu=:\nu_{\complement B}$. Therefore, we can disintegrate
    $\mu_n$ with respect to $\nu_{\complement B}$ and get
    \begin{align*}
      \mu_n(\dd\omega) ~=~ (\mu_n)_{\omega_{\complement
          B}}(\dd\omega_B)\nu_{\complement B}(\dd\omega_{\complement
        B})\;.
    \end{align*}
    Denote by $(q_\omega)_\omega$ the disintegration of $q$ with
    respect to $\nu$ and by $\nu_{B,\omega_{\complement B}}$ the
    disintegration of $\nu$ with respect to $\nu_{\complement B}.$
    Then, we have
    \begin{align}\label{eq:mu prime disint}
      (\mu_n)_{\omega_{\complement B}}(\dd\omega_B)~=~ \int T\big(\gamma,(\omega_B,\omega_{\complement B})\big)\
      q_{\omega_B,\omega_{\complement B}}(\dd\gamma)\
      \nu_{B,\omega_{\complement B}}(\dd\omega_B)\;.
    \end{align}
    By disintegration we have
    \begin{align}\label{eq:ent disint}
      \ent(\mu_n|\pi)=\int \ent((\mu_n)_{\gamma_{\complement B}}|\pi_B)\
      \nu_{\complement B}(\dd\gamma_{\complement B}) +
      \ent(\nu_{\complement B}|\pi_{\complement B})\;,
    \end{align}
    where $\pi_B=(r_B)_*\pi.$ By monotonicity of the entropy under
    push forward, it holds that $\ent(\nu_{\complement B}|\pi_{\complement B})\leq \ent(\nu|\pi)<\infty$.
    Thus it remains to show that the first term is finite.  We will derive an
    estimate on $\ent((\mu_n)_{\gamma_{\complement B}}|\pi_B)$ which is
    integrable w.r.t.\ $\nu_{\complement B}$ yielding the result.

    We fix $\gamma_{\complement B}=\omega_{\complement B}$ and write
    -- for notational convenience -- $(\mu_n)_{\omega_{\complement
        B}}=\theta.$ The configuration space over the set $B$ will be
    denoted by $\cs_B$. It can be decomposed into
    $\bigcup_{k\geq0}\cs_B^{(k)},$ where
    $\cs_B^{(k)}=\{\gamma\in\cs_B:\gamma(B)=k\}.$ Note that for all
    $\rho=f\pi_B\in\cP(\cs_B)$ we have
    \begin{align}\label{eq:ent over compact}
      \ent(\rho|\pi_B)
      ~=~
      \sum_{k\geq 0}\rho_k \Big[\int_{\cs^{(k)}}f_k\log f_k\dd\pi_{B,k} + \log\frac{\rho_k}{\pi_k}\Big]\;,
    \end{align}
    where for each $k$ we have set $\pi_k=\pi_B(\cs_B^{(k)})$,
    $\pi_{B,k}=\pi_k^{-1}(\pi_B)_{\llcorner \cs_B^{(k)}}$, as well as
    $\rho_k=\rho(\cs_B^{(k)})$ and $\rho_k^{-1}\rho=\pi_k^{-1}f_k\
    \pi_B$ on $\cs_B^{(k)}$. By \eqref{eq:same mass in B}, we have that
    $\theta(\gamma:\gamma(B)=k)=\nu_{B,\omega_{\complement
        B}}(\gamma:\gamma(B)=k)$ for all $k$, i.e.\
    $\theta_k=(\nu_{B,\omega_{\complement B}})_k.$ Since $\nu\in
    D(\ent)$, the formulas \eqref{eq:ent over compact} and
    \eqref{eq:ent disint} imply that 
    \begin{align*}
      \sum_k \theta_k \log \frac{\theta_k}{\pi_k}
     ~\leq~
       \ent(\nu_{B,\omega_{\complement B}}|\pi_B)~\in~
      L^1(\nu_{\complement B})\;.
    \end{align*}
    By \eqref{eq:ent over compact}, we therefore need to find a good
    estimate on $\ent(\theta_k^{-1}\theta|\pi_{B,k})$ for all $k$. Put
    $A_k:= T^{-1}(\cs^{(k)}\cup \omega_{\complement B})$. By Jensen's
    inequality and \eqref{eq:mu prime disint} we have
    \begin{align*}
      \ent&(\theta_k^{-1}\theta|\pi_{B,k})\\ &\leq \int_{A_k}
      1/\theta_k\ \ent((r_B)_*T(\gamma,(\omega_B,\omega_{\complement B}))|\pi_{B,k}) \ q_{\omega_B,\omega_{\complement B}}(d\gamma)
      \nu_{B,\omega_{\complement B}}(d\omega_B)\;.
    \end{align*}
    Hence, we need to estimate the entropy of
    $(r_B)_*T(\gamma,(\omega_B,\omega_{\complement B}))$ which is a
    random $k$-point configuration, where each point of the
    configuration is uniformly distributed on a ball of radius
    $\alpha=1/(2\sqrt{n\xi(B)})$ independently of the others. Putting $\tilde m=
    m_{\llcorner B}/m(B)$ and $U_i=U_{\chi(x_i),\alpha}$ for
    $\xi(\gamma,\omega)\cap B =\sum_{i=1}^k \d_{x_i}$ we get using
    $m(B(x,r))\geq \kappa r^N$ uniformly in $x\in B$ and $ r\in
    [0,1/2]$ for some constants $\kappa$ and $N$
    \begin{align*}
      \ent((r_B)_*T(\gamma,(\omega_B,\omega_{\complement B}))|\pi_{B,k})~&=~\ent(\Pi_{i=1}^k U_i|\tilde m^{\otimes k})\\
      ~&=~
      \sum \ent(U_i|\tilde m) \leq C k (\log k + \log n ),
    \end{align*}
    for some constant $C$ depending only on $B$. Putting everything
    together we get
    \begin{align*}
      \ent(\theta|\pi_B)~&\leq~ C \sum_{k\geq 0} \pi_k k (\log k + \log
      n) + \sum_{k\geq 0} \theta_k \log \frac{\theta_k}{\pi_k}\\
      &\leq~ C' + \ent(\nu_{B,\omega_{\complement B}}|\pi_B)\;,
    \end{align*}
    which is in $ L^1(\nu_{\complement B})$ by \eqref{eq:ent disint}
    and the assumption that $\nu\in D(\ent)$. This finishes the proof.
\end{proof}

\begin{lemma}\label{lem:meas selec}
  Let $\mu,\nu\in \cP(\cs)$ with $W_2(\mu,\nu)<\infty$ and
  $q\in\Opt(\mu,\nu).$ Then there is a measurable selection
  $S:\supp(q)\to \cs_{M^2}$ of optimal matchings.
\end{lemma}

\begin{proof}
  Take $(\gamma,\omega)\in\supp(q).$ Any matching of $\gamma$ and
  $\omega$ can be identified with an element of the configuration
  space over $M^2$, denoted by $\cs_{M^2}.$ Note that the map
  assigning to $\eta$ its marginals $p_1(\eta)$ and $p_2(\eta)$ is
  measurable w.r.t.\ the vague topologies on $\cs_{M^2}$ and $\cs_M.$
  Moreover, by Lemma 4.1 (i) and (vi) of \cite{RS99} the mappings
$$G: \cs_{M^2}\to [0,\infty] \quad \eta\mapsto \int d^2(x,y) \dd \eta(x,y)$$
and
$$ \tilde F: \cs_M\times \cs_M\to [0,\infty] \quad (\gamma,\omega)\mapsto \dcs(\gamma,\omega)$$
are lower semicontinuous. Hence, the function $ F = \tilde F \circ
(p_1(\cdot),p_2(\cdot))$ is measurable w.r.t.\ the vague topology on
$\cs_{M^2}.$ (Note that we always have $F(\eta)\leq G(\eta).$) Then
the set
$$ L=\{(\gamma,\omega,\eta) : (p_1(\eta),p_2(\eta))=\eta, F(\eta)=G(\eta)\}$$
is Borel measurable. Moreover, as the set of optimal matchings of
$(\gamma,\omega)$ is closed (even compact) we can use the selection
Theorem by Kuratowski and Ryll-Nardzewski (e.g.\ \cite[Thm.\
5.2.1]{Sri98}) to get the desired map.
\end{proof}

\bibliographystyle{plain}
\bibliography{config}

\end{document}